\documentclass[leqno]{amsart}

\usepackage{amssymb,amsmath,amsthm}
\usepackage[latin1]{inputenc}
\usepackage{graphicx,color}
\usepackage[dvipsnames]{xcolor}
\usepackage{stackrel}

\usepackage[toc,page]{appendix}
\usepackage{comment}

\usepackage{geometry} 
\usepackage[english]{babel}
\usepackage[colorlinks,citecolor=blue]{hyperref}

\usepackage{tikz} 
\usepackage{pgfplots}
\usepgfplotslibrary{groupplots}
\usepgflibrary{arrows}
\usetikzlibrary{calc}
\usetikzlibrary{decorations.markings}
\usetikzlibrary[intersections]

\def \R {{\mathbb R}}
\def \N {{\mathbb N}}
\def \Z {{\mathbb Z}}
\def \C {{\mathbb C}}

\def \al {\alpha}
\def \la {\lambda}
\def \ve {\varepsilon}
\def \vt {\vartheta}
\def \vp {\varphi}

\newcommand{\sphere}{\mathbb{S}^{d-1}} 
\newcommand{\cerchio}{\mathbb{S}^{1}} 

\newcommand{\sset}{\subseteq} 
\newcommand{\uguale}{\stackrel{.}{=}} 
\newcommand{\wconv}{\rightharpoonup} 
\newcommand{\ninf}{\stackrel[n\to+\infty]{}{\longrightarrow}} 

\newtheorem{theorem}{Theorem}[section]
\newtheorem{lemma}[theorem]{Lemma}

\newtheorem{proposition}[theorem]{Proposition}
\newtheorem{remark}[theorem]{Remark}
\newtheorem{corollary}[theorem]{Corollary}
\newtheorem*{mtheorem}{Main Theorem}

\title{Minimal collision arcs asymptotic to central configurations}
\author{Vivina Barutello, Gian Marco Canneori and Susanna Terracini}

\address{Dipartimento di Matematica ``G. Peano''
	\newline\indent
	Universit\`a degli Studi di Torino
	\newline\indent
	 Via Carlo Alberto 10, 10123 Torino, Italy\\}
\email{vivina.barutello@unito.it}
\email{gianmarco.canneori@unito.it}
\email{susanna.terracini@unito.it}

\date{\today} 



\begin{document}

\maketitle
\tikzset{-<-/.style={decoration={
			markings,
			mark=at position .6 with {\arrow{<}}},postaction={decorate}}}
\begin{abstract}
We are concerned with the analysis of finite time collision trajectories for a class of singular anisotropic homogeneous potentials of degree $-\alpha$, with $\alpha\in(0,2)$ and their lower order perturbations. It is well known that, under reasonable generic assumptions, the asymptotic normalized configuration converges to a central configuration. Using McGehee coordinates, the flow can be extended to the \emph{collision manifold} having central configurations as stationary points, endowed with their stable and unstable manifolds. We focus on the case when the asymptotic central configuration is a global minimizer of the potential on the sphere: our main goal is to show that, in a rather general setting, the local stable manifold coincides with that of the initial data of minimal collision arcs.
This characterisation may be extremely useful in building complex trajectories with a broken geodesic method. The proof takes advantage of the generalised  Sundman's monotonicity formula. 
\end{abstract}
\section{Introduction and main result}
Many papers of the recent literature are focused  on the variational properties of expanding (parabolic or hyperbolic) or collapsing trajectories for $N$-body and $N$-centre type problems (see e.g. \cite{ BTVplanar, BTV, BoDaTe,BoDaPa,BHPT,MaVe}), framing them in a Morse-theoretical perspective.
Indeed, in addition to answering natural questions about the nature of these motions, the variational approach is a fruitful tool when building complex trajectories exploiting gluing techniques (cf. \cite{BCT_prep}).  This application is the original motivation for this work, although we believe that the obtained result is interesting in itself. In order to state it in detail, we need some preliminaries on the motion near collision  for a class of singular anisotropic homogeneous potentials  of degree $-\alpha$, with $\alpha\in(0,2)$ (and their lower order perturbations). 

We consider the Newtonian system of ordinary differential equations
\begin{equation}\label{eq:an_kep}
	\ddot{x}(t)=\nabla V(x(t)),
\end{equation}
whose solutions satisfy the energy relation 
\begin{equation}\label{eq:energy}
    \frac12|\dot{x}(t)|^2-V(x(t))=h,
\end{equation}
with $h\in\R$. It is possible to reword equations \eqref{eq:an_kep}-\eqref{eq:energy} using the Hamiltonian formalism, choosing, as usual, the total energy to be the Hamiltonian function. Since we will study fixed-energy trajectories, it makes sense to restrict our discussion to the $(2d-1)$-dimensional \emph{energy shell}
	\[
	\mathcal{H}_h=\left\lbrace(q,p)\in T\R^d:\ \frac12|p|^2-V(q)=h\right\rbrace\simeq\R^{2d-1},
	\]
and thus, every solution of \eqref{eq:an_kep}-\eqref{eq:energy} can be seen as an evolving pair $(q,p)\in\mathcal{H}_h$ which solves
\begin{equation}\label{eq:ham_sys}
\begin{cases}
\dot{q}=p \\
\dot{p}=\nabla V(q).
\end{cases}
\end{equation}

Our potential  $V$ is a \emph{not too singular} perturbation of a $-\al$-homogeneous potential $S$.  To be precise, for $d\geq 2$, let us introduce a function $U\in\mathcal{C}^{2}(\sphere)$ such that
\begin{equation}\label{hyp:s_star}
\tag{$s^*0$} 
\begin{cases}
\exists\, s^*\in\sphere\ \text{s.t.}\ U(s)\geq U(s^*)>0,\ \forall s\in\sphere; \\
\exists\,\delta,\mu>0\,:\  \forall s\in\sphere\ \text{s.t.}\ |s-s^*|<\delta\ \implies U(s)-U(s^*)\geq\mu|s-s^*|^2,\;
\end{cases}
\end{equation}
and then consider a potential $V\in\mathcal{C}^1(\R^d\setminus\{0\})$ such that
\begin{equation}\label{hyp:V}
\tag{$V0$} 
\begin{cases}
V=S+W; \\
S\in\mathcal{C}^2(\R^d\setminus\{0\})\ \text{and}\ S(x)=|x|^{-\al}U(x/|x|),\ \text{for some}\ \al\in(0,2); \\
 \lim\limits_{|x|\to 0}|x|^{\al'}(W(x)+|x|\cdot|\nabla W(x)|)=0,\ \text{for some}\ \al'<\al.
\end{cases}
\end{equation}
 Here, $S$ has a singularity in the origin and it represents a generalization of the anisotropic Kepler potential introduced by Gutzwiller (\cite{Gutzwiller73,Gutzwiller77,Gutzwiller81}). On the other hand, the perturbation term $W$ morally vanishes when $|x|\to 0$. In particular, recalling that a \emph{central configuration} for $S$ is a unitary vector which is a critical point of the restriction of $S$ to the sphere, the assumptions \eqref{hyp:s_star} on $U$ state that $s^*$ is a globally minimal non-degenerate central configuration for $S$.

We are concerned with the behaviour of those trajectories which collide with the attraction centre in finite time (\emph{collision solutions}). It is well known that, as $|q(t)|\to 0$, the normalized configuration $q(t)/|q(t)|$ has infinitesimal distance from the set of central configurations of $S$. In particular, if this set is discrete, any collision trajectory admits a limiting central configuration $\hat{s}\in\sphere$ (see for instance \cite{BFT2,FT2003,BTV,Sund1913,WintBook}), that is
\begin{equation}\label{eq:limit}
\lim\limits_{t\to T}\frac{q(t)}{|q(t)|}=\hat{s},
\end{equation}
for some $T>0$. Given a central configuration $\hat{s}\in\sphere$ for $S$, we define the set of initial conditions for \eqref{eq:ham_sys} in $\mathcal{H}_h$ which evolve to collision with limiting configuration $\hat{s}$ 
\[
\mathcal{S}_h(\hat{s})=\{(q,p)\in\mathcal{H}_h:\ \mbox{the solution of}\ \eqref{eq:ham_sys},\ \mbox{with}\ q(0)=q,\ p(0)=p,\ \mbox{satisfies}\ \eqref{eq:limit}\}.
\]
The corresponding motion is termed $\hat{s}$-asymptotic trajectory and we want to remark that the above set is non-empty, since the $\hat{s}$-homothetic  trajectory with energy $h$ is entirely contained in it. 

Following McGehee (\cite{McG1974,McG1981}), it is possible to prove the following result (see Sections \ref{sec:coll_man}-\ref{sec:st_man} for a step-by-step proof in the planar unperturbed case), in order to give a dynamical interpretation of the set $\mathcal{S}_h(s^*)$ when $s^*\in\sphere$ satisfies \eqref{hyp:s_star}.

    \begin{lemma}\label{lem:mcgehee} Given $h\in\R$, consider a potential $V\in\mathcal{C}^1(\R^d\setminus\{0\})$  and $s^*\in\sphere$ satisfying respectively  \eqref{hyp:V} and \eqref{hyp:s_star}. Then, there exists a diffeomorphism
    \[
	\begin{aligned}
	\phi\colon&\mathcal{H}_h\to[0,+\infty)\times T\sphere \\
	 &(q,p)\mapsto\phi(q,p)=(r,s,u)	
    \end{aligned}
    \]
    such that, for some $\mathcal{C}^2$-vector field $F\colon[0,+\infty)\times T\sphere\to[0,+\infty)\times T\sphere$ and a certain time rescaling $\tau=\tau(t)$, considering the dynamical system (where ``\,$'$\,'' stands for the derivative with respect to $\tau$)
    \begin{equation}\label{eq:vec}
    (r',s',u')=F(r,s,u),
    \end{equation}
    we have:
    \begin{itemize}
    \item[$(i)$] to a solution $(q,p)=(q(t),p(t))_{t\in[0,T)}\sset\mathcal{H}_h$ of \eqref{eq:ham_sys} there corresponds a solution $(r,s,u)=(r(\tau),s(\tau),u(\tau))_{\tau\geq 0}\sset[0,+\infty)\times T\sphere$ of \eqref{eq:vec}$;$
    \item[$(ii)$] $(0,s^*,0)$ is a hyperbolic equilibrium point for \eqref{eq:vec}$;$
    \item[$(iii)$] there exists a $d$-dimensional stable manifold $\mathcal{W}^S$ for $(0,s^*,0)$, which is locally the graph of a $\mathcal{C}^2$-function $\Psi\colon\mathcal{U}\to T_s\sphere$, where $\mathcal{U}\sset[0,+\infty)\times\sphere$ is a sufficiently small neighbourhood of $(0,s^*)$ and $\Psi(0,s^*)=0$.
    \end{itemize} 
    In other words, defining
    \[
    \mathcal{W}_{loc}^S=\mathcal{W}^S\cap(\mathcal{U}\times\Psi(\mathcal{U})),
    \]    
    it turns out that
    \begin{itemize}
    \item[$(iv)$] in a neighbourhood of the origin, $\mathcal{S}_h(s^*)$ corresponds to $\mathcal{W}_{loc}^S$ through the diffeomorphism $\phi$, so that a $s^*$-asymptotic collision trajectory will be represented by an orbit contained in $\mathcal{W}_{loc}^S$.
    \end{itemize}
    \end{lemma}

Our goal is to establish a link between orbits contained in $\mathcal{W}_{loc}^S$ and collision trajectories which minimize the geometric functional naturally associated with the Hamiltonian system. For this reason, let us introduce the Jacobi-length functional
\[
\mathcal{L}_h(y)=\int_0^T|\dot{y}|\sqrt{h+V(y)},
\]
for $y\in H^1([0,T];\R^d)$ such that $|\dot{y}|>0$ and $h+V(y)>0$. It is well known that a critical point $y$ of $\mathcal{L}_h$ corresponds to a classical solution on $(0,T)$ of \eqref{eq:ham_sys} in $\mathcal{H}_h$ for a certain $T>0$, if $|y(t)|\neq 0$ for every $t\in(0,T)$ (see for instance \cite{ArBook,KnBook,MoMoVe2012}). 

In particular, for a properly chosen $\bar{r}=\bar{r}(h)>0$ and for $q\in B_{\bar{r}}=B_{\bar r}(0)$, introducing the set of collision paths
\[
H_{coll}^{q}=\{y\in H^1([0,T];\R^d):\ y(0)=q,\ y(T)=0,\ |y(t)|<|q|,\ t\in(0,T)\},
\] 
contrary to the case $\alpha\geq2$, when $\mathcal{L}_h$ is never finite on collisions,  when  $\alpha\in(0,2)$, we are able to find at least a minimizer for the Jacobi lenght in the above space. Such a minimizer is not necessarily unique; indeed, any of these minimal paths is associated with the starting velocity $\dot{y}(0)$ of the trajectory. This leads to the construction of the multivalued map 
\[
\begin{aligned}
\mathcal{F}_h\colon &B_{\bar{r}}\to\mathcal{P}(T_{q}\R^d) \\
           &q\mapsto  \mathcal{F}_h(q)=\left\lbrace\dot{y}(0):\ y=\arg\min\limits_{H_{coll}^{q}}\mathcal{L}_h\right\rbrace
\end{aligned}
\]
in which $\mathcal{F}_h(q)$ represents the set of all the initial velocities for which a minimal collision arc exists.

Now, in the fashion of Lemma \ref{lem:mcgehee}, without loss of generality, we can assume that $\mathcal{U}=[0,\bar{r})\times B_{\bar{\delta}}(s^*)$ for some $\bar{\delta}>0$, so that
\[
\mathcal{W}_{loc}^S=\mathcal{W}_{loc}^S(\bar{r},\bar{\delta}).
\]
In this way, our main result consists in showing that, if $\bar{r}$ and $\bar{\delta}$ are sufficiently small, a collision minimizer starting at $q\in B_{\bar{r}}$, with $q/|q|\in B_{\bar{\delta}}(s^*)$, is actually unique and its $\phi$-corresponding orbit is entirely contained in $\mathcal{W}_{loc}^S$. This means that, for such starting points, the set $\mathcal{F}_h(q)$ is not only a singleton, but it verifies $\phi(q,\mathcal{F}_h(q))\in\mathcal{W}_{loc}^S$. For this reason, it makes sense to introduce another \emph{local set} in the phase space, which is \emph{spanned} by all the unique minimizers above mentioned
\[
\mathfrak{M}_h(\bar{r},\bar{\delta})=\left\lbrace\phi(q,\mathcal{F}_h(q)):\ q\in B_{\bar{r}},\ q/|q|\in B_{\bar{\delta}}(s^*)\right\rbrace,
\]
and to state our core result in this way:

\begin{mtheorem}
	Given $h\in\R$, consider a potential $V\in\mathcal{C}^1(\R^d\setminus\{0\})$ and $s^*\in\sphere$ verifying respectively \eqref{hyp:V} and \eqref{hyp:s_star}. Then, there exist $\bar{r}=\bar{r}(h)>0$ and $\bar{\delta}=\bar{\delta}(s^*)>0$ such that
	\[
	\mathcal{W}_{loc}^S(\bar{r},\bar{\delta})=\mathfrak{M}_h(\bar{r},\bar{\delta}).
	\]
\end{mtheorem}

\begin{remark} The assumption \eqref{hyp:s_star} that the minimal central configuration is non-degenerate, though stringent, holds generically. It can be easily lifted in some particular situations, for example in the case of the $-\alpha$-homogeneous $N$-body problem, that is when
\[
V(q_1,\cdots,q_N)=\sum_{i\neq j}\dfrac{m_im_j}{|q_i-q_j|^{\alpha}}\;.
\]
In this case the potential is invariant under common rotations of all the bodies and obviously no central configuration can be non-degenerate. However, our main result still holds true under the assumption of non-degeneration of the $\mathcal{SO}(d)$-orbit of the minimal central configuration under examination. Indeed, using again McGehee change of coordinates and extending the flow on the collision manifold, Lemma \ref{lem:mcgehee} can be rephrased in terms of a normally invariant manifold of stationary points endowed with their stable and unstable (local) manifolds. Given this alteration, the statement and proof easily follow.
\end{remark}

For the sake of a better comprehension and visualization of the proofs, we will carry out our work in a simplified case, which can be easily generalized to the setting introduced above. In particular, from now on we will take into account a planar anisotropic Kepler problem as proposed in \cite{BTVplanar} and we will work in negative energy shells, i.e., we will assume
\begin{itemize}
	\item $d=2$;
	\item $W\equiv 0$;
	\item $h<0$.
\end{itemize}
Useful complementary material needed for the proof in the more general setting will be provided in Section \ref{sec:generalizations}. The paper is organised as follows: Section \S\ref{sec:coll_man} introduces the collision manifold for the planar case and recalls the main features of the extended flow, whereas \S\ref{sec:st_man}  is devoted to the analysis of the extended flow near its critical points. The object of Section \S\ref{sec:var} are Bolza minimizing arcs and their properties, while the Main Theorem will be eventually proved in \S\ref{sec:main} in the unperturbed and planar case, whereas in \S \ref{sec:generalizations} we will discuss the modifications needed to cover the perturbed $d$-dimensional case.

\section{The collision manifold for the planar problem}\label{sec:coll_man}

As aforementioned, we will develop this and the following sections working on the plane and with an unperturbed potential $V$. The following construction, which is the two-dimensional version of Lemma \ref{lem:mcgehee}, exploits a technique firstly introduced by R. McGehee in the study of the collinear 3-body problem (\cite{McG1974, McG1981}) and furthermore employed by Devaney and others for the anisotropic Kepler problem (\cite{DevInvMath1978, DevProgMath1981,Deva1982,SaaHulk81}). Exploiting a space-time change of coordinates, this method consists in attaching a \emph{collision manifold} to the phase space, where the flow can be extended in a suitable way,  having central configurations as stationary points, endowed with their stable and unstable manifolds. In particular, a very similar approach with possibly different time parameterization can be found in \cite{BTVplanar,BTV,HuYu2018}. We shall follow here the Devaney's approach (\cite{DevInvMath1978}). For our purposes, for a point $x\in\R^2$ it makes sense to introduce polar coordinates $x=(q_1,q_2)=(r \cos\vt, r \sin\vt)$, where
\[
r=\sqrt{q_1^2+q_2^2}\geq 0,\qquad\vt=\arctan(q_2/q_1)\in[0,2\pi).
\]
In this way, any $-\al$-homogeneous potential $V\in\mathcal{C}^2(\R^2\setminus\{0\})$ can be written as 
\[
V(x)=r^{-\al}U(\vt),
\]
where $U\in\mathcal{C}^2(\cerchio)$, $U>0$ and
\[
U(\vt)=V(\cos\vt,\sin\vt).
\]         
\paragraph{\textbf{Hypotheses on }$\boldsymbol{V}$:} In this setting, the original assumptions \eqref{hyp:V}-\eqref{hyp:s_star} reduce respectively to:
\begin{equation}\label{hyp:V1}
\tag{$V1$}
\begin{cases}
V\in\mathcal{C}^2(\R^2\setminus\{0\}); \\
V(x)=|x|^{-\al}U(x/|x|),\ \text{with}\ \al\in(0,2)\ \text{and}\ U\in\mathcal{C}^2(\cerchio),
\end{cases}
\end{equation}
and
\begin{equation}\label{hyp:vt_star}
	\tag{$U1$}
	\exists\,\vt^*\in\cerchio\ \text{s.t.}\ U(\vt)\geq U(\vt^*)>0\ \forall\,\vt\in\cerchio\ \text{and}\ U''(\vt^*)>0.
\end{equation}

With these notations, we study the motion and energy equations in the plane
\begin{equation}\label{pb:an_kep}
\begin{cases}
\ddot{x}(t)=\nabla V(x(t)) \\ 
\frac{1}{2}|\dot{x}(t)|^2-V(x(t))=h,
\end{cases}
\end{equation}
with $h<0$. As usual, the conservation of energy forces every solution of \eqref{pb:an_kep} to be included into the \emph{Hill's region}
\[
\mathcal{R}_h=\left\lbrace x\in\R^2\setminus\{0\}:\ V(x)+h\geq 0\right\rbrace.
\]
Now, since
\[
\nabla r=r^{-1}(q_1,q_2),\qquad \nabla\vt=r^{-2}(-q_2,q_1),
\]
we can compute
\[
\nabla V(x)=r^{-\al-2}\left[-\al U(\vt)(q_1,q_2)+U'(\vt)(-q_2,q_1)\right].
\]
In this way, introducing the momentum vector $(p_1,p_2) = (\dot q_1,\dot q_2)$, we can  rewrite equations \eqref{pb:an_kep} as
\begin{equation}\label{eq:polar}
\begin{cases}
\dot q_1 = p_1 \\
\dot q_2 = p_2 \\
\dot p_1 = {r^{-\al-2}}\left[ -U'(\vt)q_2 -\al U(\vt)q_1 \right] \\
\dot p_2 = {r^{-\al-2}}\left[ U'(\vt)q_1 -\al U(\vt)q_2 \right],
\end{cases}
\end{equation}
and
\[
\frac12 \left( p_1^2 + p_2^2\right) - r^{-\al}U(\vt)=h.
\]
If we are not on the boundary of $\mathcal{R}_h$, we have that $|p|\neq0$ and so, for every solution of \eqref{eq:polar}, we can find smooth functions
$z>0$ and $\vp\in[0,2\pi)$ in such a way that
$p_1=r^{-\al/2}z\cos \vp$, $p_2=r^{-\al/2}z\sin\vp$, choosing
\begin{equation}\label{eq:zeta}
z = \sqrt{2U(\vt)+2hr^\al}.
\end{equation}
By standard calculations, equations \eqref{eq:polar} become
\begin{equation}\label{eq:sing}
\begin{cases}
\dot r   = r^{-\al/2}z \cos (\vp-\vt) \\
\dot \vt = r^{-1-\al/2}z\sin (\vp-\vt) \\
\dot z   = r^{-1-\al/2}\left[U'(\vt)\sin(\vp-\vt)+\al h r^\al\cos(\vp-\vt)\right] \\
\dot \vp = \frac{1}{z} r^{-1-\al/2} \left[ U'(\vt)\cos(\vp-\vt) +\al U(\vt) \sin(\vp-\vt)\right]
\end{cases}
\end{equation}
and this system has a singularity when $r=0$, which indeed corresponds to the collision set $\{0\}\sset\R^2$ of problem \eqref{pb:an_kep}. Introducing a new time variable $\tau$ which verifies
\begin{equation}\label{eq:time_sc}
\frac{dt}{d\tau} = zr^{1+\al/2},
\end{equation}
the singularity of \eqref{eq:sing} can be \emph{removed} in order to extend the vector field to the singular boundary $\{r=0\}$. The effect of this rescaling is to \emph{blow-up} the instant of an eventual collision, so that the particle will \emph{virtually never reach} the singularity. In this way, we can rewrite \eqref{eq:sing} as (here ``\,$'$\,'' denotes the derivative with respect to $\tau$)
\[
\begin{cases}
r'   = r z^2 \cos (\vp-\vt) \\
\vt' = z^2 \sin (\vp-\vt) \\
z'   = z \left[U'(\vt)\sin(\vp-\vt)+\al h r^\al\cos(\vp-\vt)\right] \\
\vp' =U'(\vt)\cos(\vp-\vt) +\al U(\vt)\sin(\vp-\vt).
\end{cases}
\]
Moreover, the conservation of energy, together with definition \eqref{eq:zeta}, allows us to eliminate the variable $z$ from the system, and thus to consider the 3-dimensional system
\begin{equation}\label{eq:mcgehee}
\begin{cases}
r'   = 2r(U(\vt)+hr^\al) \cos (\vp-\vt) \\
\vt' = 2(U(\vt)+hr^\al) \sin (\vp-\vt)  \\
\vp' =U'(\vt)\cos(\vp-\vt)+\al U(\vt)\sin(\vp-\vt)
\end{cases}
\end{equation}
which we shortly denote by $(r',\vt',\vp')=F(r,\vt,\vp)$, with $F\colon[0,+\infty)\times T\cerchio\to[0,+\infty)\times T\cerchio$. Since $r'=0$ when $r=0$, the boundary $\{r=0\}$ is an invariant set for the above system. In other words, we can restrict the vector field $F$ to $\{r=0\}$  and study the independent dynamical system
\begin{equation}\label{eq:coll_man}
\begin{cases}
\vt' = U(\vt)\sin (\vp-\vt) \\
\vp' =U'(\vt)\cos(\vp-\vt) +\al U(\vt) \sin(\vp-\vt)
\end{cases}
\end{equation}
which defines a 2-dimensional \emph{collision manifold} and whose stationary points are
\[
(\hat{\vt},\hat{\vt}+k\pi),\ \mbox{with}\ k\in\Z\ \mbox{and}\ U'(\hat{\vt})=0.
\]
If we denote by $JF|_{\{r=0\}}$ the Jacobian matrix of the vector field associated to \eqref{eq:coll_man} and we evaluate it at $(\hat\vt,\hat\vt+k\pi)$, we obtain
\[
JF|_{\{r=0\}}(\hat{\vt},\hat{\vt}+k\pi)=\cos(k\pi)U(\hat{\vt})\begin{pmatrix}
-2 & 2 \\
\frac{U''(\hat{\vt})}{U(\hat{\vt})}-\al & \al
\end{pmatrix},
\]
whose eigenvalues are
\[
\mu^\pm=\frac{1}{2}\cos(k\pi)U(\hat{\vt})\left\lbrace\al-2\pm\left[(\al-2)^2+8\frac{U''(\hat{\vt})}{U(\hat{\vt})}\right]^{\frac{1}{2}}\right\rbrace.
\]
For the sake of completeness, we present here a full characterization of the equilibrium points of \eqref{eq:coll_man}, depending on the value of $U''(\hat\vt)$.

\emph{For an equilibrium point $(\hat\vt,\hat\vt+k\pi)$ of \eqref{eq:coll_man} we can have:
\begin{itemize}
	\item if $U''(\hat{\vt})<-\frac{(\al-2)^2}{8}U(\hat{\vt})$, then $\mu^\pm\in\C\setminus\R$ and thus
	\begin{itemize}
		\item if $k$ is even then $(\hat{\vt},\hat{\vt}+k\pi)$ is a \emph{sink}$;$
		\item if $k$ is odd then $(\hat{\vt},\hat{\vt}+k\pi)$ is a \emph{source}$;$
	\end{itemize}
	\item if $U''(\hat{\vt})=-\frac{(\al-2)^2}{8}U(\hat{\vt})$, then $\mu^-=\mu^+\in\R$ and thus
	\begin{itemize}
		\item if $k$ is even then $(\hat{\vt},\hat{\vt}+k\pi)$ is \emph{asymptotically stable}$;$
		\item if $k$ is odd then 
		$(\hat{\vt},\hat{\vt}+k\pi)$ is \emph{unstable}$;$
	\end{itemize}
	\item if $U''(\hat{\vt})>-\frac{(\al-2)^2}{8}U(\hat{\vt})$, then 
	\begin{itemize}
		\item if $U''(\hat{\vt})>0$, then $\mu^-\cdot\mu^+<0$ and thus $(\hat{\vt},\hat{\vt}+k\pi)$ is a \emph{saddle}$;$
		\item if $U''(\hat{\vt})=0$, then one of the eigenvalues is zero and thus
		\begin{itemize}
			\item if $k$ is even then $(\hat{\vt},\hat{\vt}+k\pi)$ is a \emph{stable degenerate node}$;$
			\item if $k$ is odd then $(\hat{\vt},\hat{\vt}+k\pi)$ is a \emph{unstable degenerate node}$;$    		
		\end{itemize}
		\item if $0>U''(\hat{\vt})>-\frac{(\al-2)^2}{8}U(\hat{\vt})$, then $\mbox{sign}(\mu^-)=\mbox{sign}(\mu^+)$ and thus
		\begin{itemize}
			\item if $k$ is even then $(\hat{\vt},\hat{\vt}+k\pi)$ is a \emph{stable 2-tangents node}$;$
			\item if $k$ is odd then $(\hat{\vt},\hat{\vt}+k\pi)$ is a \emph{unstable 2-tangents node}.    	
		\end{itemize}
	\end{itemize}
\end{itemize}
}
Now, we assume again \eqref{hyp:V1}-\eqref{hyp:vt_star} and so, in particular, $(\vt^*,\vt^*+k\pi)$ is a saddle equilibrium point for \eqref{eq:coll_man}. Coming back to the 3-dimensional system \eqref{eq:mcgehee}, the Jacobian of $F$ evaluated in the stationary points $(0,\vt^*,\vt^*+k\pi)$ is
\[
JF(0,\vt^*,\vt^*+k\pi)=U(\vt^*)\cos(k\pi)\begin{pmatrix}
2 & 0 & 0 \\
0 & -2 & 2 \\
0 & \frac{U''(\vt^*)}{U(\vt^*)}-\al & \al
\end{pmatrix}.
\]
In this way, we note that the $r$-eigenvalue is always non-zero, i.e.
\begin{itemize}
	\item if $k$ is odd the orbit \emph{enters} in the collision manifold;
	\item if $k$ is even the orbit \emph{leaves} the collision manifold.
\end{itemize}
Since we are interested in studying the behaviour of a trajectory which approaches the singularity, we focus our attention on the case in which $k$ is odd. With the choice of $k=1$, the Jacobian becomes
\[
JF(0,\vt^*,\vt^*+\pi)=\begin{pmatrix}
-2U(\vt^*) & 0 & 0 \\
0 & 2U(\vt^*) & -2U(\vt^*) \\
0 & \al U(\vt^*)-U''(\vt^*) & -\al U(\vt^*)
\end{pmatrix},
\]
the eigenvalues are
\[
\begin{aligned}
\la_r&= -2U(\vt^*)<0 \\
\la^\pm&= \frac{2-\al}{2}U(\vt^*)\pm\frac12\sqrt{(2-\al)^2U(\vt^*)^2+8U(\vt^*)U''(\vt^*)}\gtrless0 
\end{aligned}
\]
and the relative eigendirections are
\[
\begin{aligned}
v_r&=(1,0,0) \\
v^\pm&=\left(0,1,\frac12 +\frac\al4 \pm\frac14\sqrt{(2-\al)^2+8\frac{U''(\vt^*)}{U(\vt^*)}}\right).
\end{aligned}
\]
\begin{remark}
The orbits of the stable manifold associated to the equilibrium point $(0,\vt^*,\vt^*+\pi)$ will enter in the collision manifold being tangent to $v_r$ or $v^-$, depending on the sign of the quantity 
\[
\frac{U''(\vt^*)}{U(\vt^*)}-(4-\al)
\]
(for instance, if negative, the tangency will be with respect to $v^-$). In particular, for the classical Kepler problem in which $\al=1$, $U(\vt^*)=1$ and $U''(\vt^*)=0$, the above quantity is always negative.
\end{remark}

\section{The stable manifold}\label{sec:st_man}

In the previous section we have shown the existence of an invariant set for \eqref{eq:mcgehee}, the \emph{collision manifold} $\{r=0\}$. Moreover, from the linearization of the vector field, it follows that the only way for a point in the phase space to evolve \emph{entering} in $\{r=0\}$ is to belong to the stable set of an equilibrium point $(0,\hat{\vt},\hat{\vt}+k\pi)$, with $U'(\hat{\vt})=0$ and $k$ odd. For this reason, in this section we focus our attention on the study of the \emph{stable manifold} of such equilibrium points, with the not restrictive choice of $k=1$. We start our analysis with the case $h=0$, which presents a simplified and clearer structure.

\subsection{Collision orbits for $\boldsymbol{h=0}$}

In this setting, system \eqref{eq:mcgehee} reads
\begin{equation}\label{eq:3d_zero}
\begin{cases}
r'=2rU(\vt)\cos(\vp-\vt) \\
\vt'=2U(\vt)\sin(\vp-\vt) \\
\vp'=U'(\vt)\cos(\vp-\vt)+\al U(\vt)\sin(\vp-\vt)
\end{cases}
\end{equation}
and the set
\[
\{(r,\hat{\vt},\hat{\vt}+\pi):\,r\geq0,\ U'(\hat{\vt})=0\}
\]
is invariant for the system and it gathers all the collision $\hat\vt$-homothetic trajectories. Once $\hat\vt\in\cerchio$ critical point for $U$ is fixed, such a set reduces to a ray which enters the collision manifold in the equilibrium point $(0,\hat\vt,\hat\vt+\pi)$.

\begin{lemma}\label{lem:stable_manif_0}
	Assume \eqref{hyp:V1}-\eqref{hyp:vt_star}. Then, there exist $\delta>0$, a stable manifold $W^s$ for the equilibrium point $(\vt^*,\vt^*+\pi)$ of \eqref{eq:coll_man} and a $\mathcal{C}^2$-function $\Psi\colon(\vt^*-\delta,\vt^*+\delta)\to\cerchio$ such that $\Psi(\vt^*)=\vt^*+\pi$ and for every $\vt\in(\vt^*-\delta,\vt^*+\delta)$
	\begin{itemize}
		\item $(\vt,\vp)\in W^s$ if and only if $\vp=\Psi(\vt)$;		 
		\item $\vp<\vt+\pi$ if $\vt\in(\vt^*-\delta,\vt^*)$ $[$resp. $\vp>\vt+\pi$ if $\vt\in(\vt^*,\vt^*+\delta)$$]$.
	\end{itemize}
    Moreover, $\R_0^+\times W^s$ is the stable manifold of the equilibrium point $(0,\vt^*,\vt^*+\pi)$ for system \eqref{eq:3d_zero}.
\end{lemma}
\begin{proof}
	We firstly analyse the 2-dimensional system \eqref{eq:coll_man}, keeping in mind the eigendirections  $v^+$ and $v^-$ computed in the previous section. From the \emph{Stable Manifold Theorem} (see for instance \cite{HirSmaDev},\cite{Teschl_ode}) we have that there exist $W^u,W^s$ $\mathcal{C}^2$-curves on the collision manifold $\{r=0\}$ such that
	\begin{itemize}
		\item $(\vt^*,\vt^*+\pi)\in W^u,W^s$;
		\item $W^s$ is tangent to $v^-$ and $W^u$ is tangent to $v^+$ in $(\vt^*,\vt^*+\pi)$;
		\item for every $(\vt^+,\vp^+)\in W^u$ and $(\vt^-,\vp^-)\in W^s$ we have
		\[
		\lim_{\tau\to\pm\infty}(\vt^\pm(\tau),\vp^\pm(\tau))=(\vt^*,\vt^*+\pi).
		\]
	\end{itemize}	
	In particular, since $(\vt^*,\vt^*+\pi)$ is a hyperbolic equilibrium point, $W^s$ is locally the graph of a $\mathcal{C}^2$-curve $\vp=\vp(\vt)$, i.e.
	\[
	W^s=\left\lbrace (\vt,\vp(\vt)):\ \vt\in(\vt^*-\delta,\vt^*+\delta)\ \mbox{with}\ \delta>0,\ \vp\in\mathcal{C}^2,\ \vp(\vt^*)=\vt^*+\pi\right\rbrace,
	\]
	which is tangent to $v^-$ in $(\vt^*,\vt^*+\pi)$. This is a consequence of the fact that the \emph{local stable manifold} has the same dimension of the \emph{stable eigenspace} (\emph{Hartman-Grobman Theorem}, see for instance \cite{Teschl_ode}).
	
	Now, since the second and third equations of system \eqref{eq:3d_zero} are uncoupled for every $r\geq 0$, if we consider the set
	\[
	\R_0^+\times W^s=\{(r,\vt,\vp):\ r\geq0,\ (\vt,\vp)\in W^s\},
	\]
    defining the flow associated to \eqref{eq:3d_zero} as $\Phi^\tau=\Phi^\tau(r,\vt,\vp)$, we have that for every $(r,\vt,\vp)\in\R_0^+\times W^s$
	\[
	\lim_{\tau\to+\infty}\Phi^\tau(r,\vt,\vp)=(0,\vt^*,\vt^*+\pi).
	\]
	Indeed, $\Phi^\tau(r,\vt,\vp)\in\R_0^+\times W^s$ for every $\tau>0$,  since $\pi_{r=0}F_0(\Phi^\tau(r,\vt,\vp))$ is tangent to $W^s$ for every $\tau>0$, where $F_0$ represents the vector field associated to \eqref{eq:3d_zero}.
\end{proof}

\subsection{Collision orbits for $\boldsymbol{h<0}$}

When $h<0$, we come back again to system \eqref{eq:mcgehee}, which we recall here for the reader's convenience
\[
\begin{cases}
r'=2r(U(\vt)+hr^\al)\cos(\vp-\vt) \\
\vt'=2(U(\vt)+hr^\al)\sin(\vp-\vt) \\
\vp'=U'(\vt)\cos(\vp-\vt)+\al U(\vt)\sin(\vp-\vt).
\end{cases}
\]
The collision manifold $\{r=0\}$ is exactly the same of the zero-energy system and we still have the invariance of the collision homothetic trajectories set
\[
\{(r,\hat{\vt},\hat\vt+\pi):\ r\geq0,\ U'(\hat\vt)=0\}.
\]
However, we point out that for $h<0$ the set $\R_0^+\times W^s$ is not the stable manifold for $(0,\hat{\vt},\hat\vt+\pi)$. Beside that, assuming \eqref{hyp:V1}-\eqref{hyp:vt_star}, the hyperbolicity of the fixed point $(0,\vt^*,\vt^*+\pi)$ still guarantees the existence of a 2-dimensional stable manifold $\mathcal{W}^s$, which contains the homothetic trajectories and the 1-dimensional stable manifold $W^s$.

Now, the dynamical systems \eqref{eq:3d_zero} and \eqref{eq:mcgehee} share the same linearization with respect to the equilibrium point $(0,\vt^*,\vt^*+\pi)$. These means that, below some $r^*>0$, they are topologically equivalent. In particular, we can imagine $\mathcal{W}^s\cap\{r<r^*\}$ as a $\mathcal{C}^2$ $h$-deformation of $[0,r^*)\times W^s$, in which the 1-dimensional components $\{0\}\times W^s$ and $(0,r^*)\times\{\vt^*\}\times\{\vt^*+\pi\}$ stay always fixed.

As a consequence of these discussions, we deduce the following analytic and geometric description of $\mathcal{W}^s$ in a neighbourhood of the equilibrium point, which locally generalizes Lemma \ref{lem:stable_manif_0}.

\begin{lemma}\label{lem:stable_manif}
	Assume \eqref{hyp:V1}-\eqref{hyp:vt_star}. Given $h<0$, there exist $r_{loc}>0$, $\delta_{loc}>0$  and a $\mathcal{C}^2$-function $\Psi\colon[0,r_{loc})\times(\vt^*-\delta_{loc},\vt^*+\delta_{loc})\to\cerchio$ such that $\Psi(0,\vt^*)=\vt^*+\pi$ and for every $(r,\vt)\in(0,r_{loc})\times(\vt^*-\delta_{loc},\vt^*+\delta_{loc})$
	\begin{itemize}
		\item $(r,\vt,\vp)\in\mathcal{W}^s$ if and only if $\vp=\Psi(r,\vt)$;		 
		\item $\vp<\vt+\pi$ if $\vt\in(\vt^*-\delta_{loc},\vt^*)$ $[$resp. $\vp>\vt+\pi$ if $\vt\in(\vt^*,\vt^*+\delta_{loc})$$]$.
	\end{itemize}   
   In other words, we have just characterized locally $\mathcal{W}^s$ as the graph of a function $\Psi$
   \[
   \mathcal{W}_{loc}^s\uguale\left\lbrace (r,\vt,\Psi(r,\vt)):\,r\in[0,r_{loc}),\ \vt\in(\vt^*-\delta_{loc},\vt^*+\delta_{loc})\right\rbrace\sset\mathcal{W}^s.
   \]
\end{lemma}

\begin{remark}\label{rem:top_conj}
	To better understand the meaning of the previous lemma, we can refer to the configurations space the behaviour of a point evolving in $\mathcal{W}_{loc}^s$ and eventually entering in the collision manifold. In particular, Lemma \ref{lem:stable_manif} guarantees the existence of a cone
	\[
	\mathcal{C}=\{q=(q_1,q_2)\in\R^2:\ |q|\leq r_{loc},\ \arctan(q_2/q_1)\in(\vt^*-\delta_{loc},\vt^*+\delta_{loc})\}
	\]
	such that, for every trajectory which starts in $\mathcal{C}$ and reaches the collision being tangent to $\vt^*$, it never leaves $\mathcal{C}$. We would stress that this confinement result is strictly addressed to those orbits which, in the phase space, are contained in $\mathcal{W}_{loc}^s$. Indeed, every collision orbit that is not tangent to $\vt^*$ in the origin is not necessarily contained in the cone $\mathcal{C}$.
\end{remark}

\begin{figure}
	\centering
	\begin{tikzpicture}
	\coordinate (O) at (0,0);
	\coordinate (x) at (0,8);
	\coordinate (-delta) at (-5.3,8.48);    	
	\coordinate (+delta) at (5.3,8.48); 
	\coordinate (p0) at (-4.23,6.78);
	\coordinate (p0m) at (4.23,6.78);
	\coordinate (p1) at (-3.75,7.06);
	\coordinate (p1m) at (3.75,7.06);
	\coordinate (p2) at (-3.25,7.30);
	\coordinate (p2m) at (3.25,7.30);
	\coordinate (p3) at (-2.73,7.51);
	\coordinate (p3m) at (2.73,7.51);
	\coordinate (p4) at (-2.20,7.69);
	\coordinate (p4m) at (2.20,7.69);        
	\coordinate (p5) at (-1.66,7.82);
	\coordinate (p5m) at (1.66,7.82);        
	\coordinate (p6) at (-1.11,7.92);
	\coordinate (p6m) at (1.11,7.92);        
	\coordinate (p7) at (-0.55,7.98);
	\coordinate (p7m) at (0.55,7.98);
	
	\draw [dashed] (O)--(0,10);
	\draw  (O)--(x);
	\draw[thick,red,-<-] (O)--(x);
	\draw[dashed] (O)--(-delta);
	\draw[dashed] (O)--(+delta);	
	
	\draw [thick,domain=-32:32] plot ( {8*sin(\x)},{8*cos(\x)});
	\draw [dashed,domain=-48:44] plot ( {8*sin(\x)},{8*cos(\x)});  
	
	\fill (O) circle[radius=1.5pt];
	\fill (x) circle[radius=1.5pt];
	\fill (p0) circle[radius=1.5pt];
	\fill (p0m) circle[radius=1.5pt];    	
	\fill (p1) circle[radius=1.5pt];
	\fill (p1m) circle[radius=1.5pt];
	\fill (p2) circle[radius=1.5pt];
	\fill (p2m) circle[radius=1.5pt];    	
	\fill (p3) circle[radius=1.5pt];
	\fill (p3m) circle[radius=1.5pt];
	\fill (p4) circle[radius=1.5pt];
	\fill (p4m) circle[radius=1.5pt];    	
	\fill (p5) circle[radius=1.5pt];
	\fill (p5m) circle[radius=1.5pt];    	
	\fill (p6) circle[radius=1.5pt];
	\fill (p6m) circle[radius=1.5pt];    	
	\fill (p7) circle[radius=1.5pt];
	\fill (p7m) circle[radius=1.5pt];
	
	\node[right] at (0,-0.3) {$0$};
	\node[right] at (0,9.1) {$\vt^*$};   	
	\node[right] at (-4.9,8.1) {$\vt^*-\delta$};    	
	\node[right] at (5.1,8.1) {$\vt^*+\delta$};
	\node at (3,4) {$\mathcal{C}$};	
	\node[blue] at (-2.2,4.2) {$\mathcal{W}_{loc}^s$};
	\node[right,red] at (0.1,8.3) {$\vt^*$-\emph{hom}}; 
	
	\draw[thick,blue,-<-] (O) parabola bend (p0) (p0); 
	\draw[thick,blue,-<-] (O) parabola bend (p1) (p1);
	\draw[thick,blue,-<-] (O) parabola bend (p2) (p2);
	\draw[thick,blue,-<-] (O) parabola bend (p3) (p3);
	\draw[thick,blue,-<-] (O) parabola bend (p4) (p4);
	\draw[thick,blue,-<-] (O) parabola bend (p5) (p5);
	\draw[thick,blue,-<-] (O) parabola bend (p6) (p6);
	\draw[thick,blue,-<-] (O) parabola bend (p7) (p7);
	\draw[thick,blue,-<-] (O) parabola bend (p0m) (p0m); 
	\draw[thick,blue,-<-] (O) parabola bend (p1m) (p1m);
	\draw[thick,blue,-<-] (O) parabola bend (p2m) (p2m);
	\draw[thick,blue,-<-] (O) parabola bend (p3m) (p3m);
	\draw[thick,blue,-<-] (O) parabola bend (p4m) (p4m);
	\draw[thick,blue,-<-] (O) parabola bend (p5m) (p5m);
	\draw[thick,blue,-<-] (O) parabola bend (p6m) (p6m);
	\draw[thick,blue,-<-] (O) parabola bend (p7m) (p7m);
	
	\end{tikzpicture}
	\caption{The local stable manifold characterized in Remark \ref{rem:top_conj}.}
\end{figure}
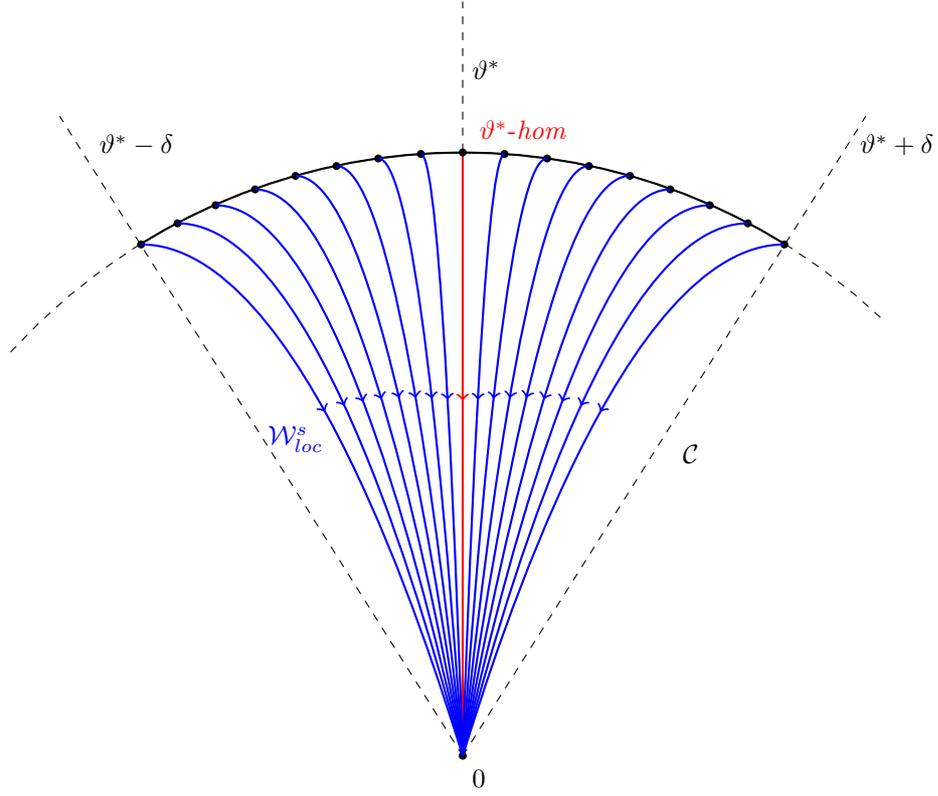

\section{Collision orbits as Bolza minimizers}\label{sec:var}

The first task of this work is to highlight the relationship which stands between the dynamical nature of this problem and our variational approach. Therefore, we present here the minimality argument which leads to the existence of a solution for \eqref{eq:an_kep}-\eqref{eq:energy} and provide further properties of this underlying variational structure of the problem. The \emph{Maupertuis' Principle} states that every critical point of a suitable functional, which could be either the Lagrange-action, the Jacobi-length or the Maupertuis' functional, if properly manipulated is a classical solution of \eqref{eq:an_kep}-\eqref{eq:energy} (see \cite{ArBook},\cite{AC-Z}). In the first part of this section, we state and prove a similar result which guarantees the existence of a trajectory which reaches the origin in finite time, once a critical point is provided. From now on, we will always consider $h<0$ fixed and assume \eqref{hyp:vt_star}-\eqref{hyp:V1}, unless differently specified.

\subsection{The Maupertuis' Principle for collision trajectories}

Given $q\in\mathcal{R}_h\setminus\{0\}$, consider the space of all the collision $H^1$-paths starting from $q$ and reaching the origin in finite time $T>0$
\[
\widehat{H}_{coll}^q=\{u\in H^1([0,T];\R^2):\ u(0)=q,\ u(T)=0\}.
\] 
Moreover, let us introduce the Maupertuis' functional $\mathcal{M}_h\colon \widehat{H}_{coll}^q\to\R\cup\{+\infty\}$ such that
\[
\mathcal{M}_h(u)=\frac12\int_0^T|\dot{u}(s)|^2\,ds\int_0^T(h+V(u(s)))\,ds
\]
which is differentiable over $\widehat{H}_{coll}^q$ and, if $\mathcal{M}_h(u)>0$, it makes sense to define the quantity
\begin{equation}\label{def:omega}
\omega=\left(\dfrac{\int_0^T(h+V(u))}{\frac12\int_0^T|\dot{u}|^2}\right)^{\frac12}>0.
\end{equation}

\begin{lemma}\label{lem:archetto}
	Let $u\in\widehat{H}_{coll}^q$ be a minimizer of $\mathcal{M}_h$, with $\mathcal{M}_h(u)>0$. Then, $u(t)\neq 0$ for every $t\in(0,T)$.
\end{lemma}
\begin{proof}
	Assume by contradiction that there exists $\tau\in(0,T)$ such that $u(\tau)=0$. Observe that the path $v(t)=u(t\cdot\tau/T)$ defined for $t\in[0,T]$ belongs to $\widehat{H}_{coll}^q$. Since the Maupertuis' functional is invariant through time rescalings, with a standard change of variable we obtain
	\[
	\begin{aligned}
	\mathcal{M}_h(u)&=\left(\frac12\int_0^\tau|\dot{u}|^2+\frac12\int_0^\tau|\dot{u}|^2\right)\left(\int_\tau^T(h+V(u))+\int_\tau^T(h+V(u))\right) \\
	&=\frac12\int_0^\tau|\dot{u}|^2\int_0^\tau(h+V(u))+[\mbox{positive terms}] \\
	&=\mathcal{M}_h(v)+[\mbox{positive terms}],
	\end{aligned}
	\]
	which is a contradiction for the minimality of $u$.
\end{proof}

\begin{theorem}\label{thm:maupertuis}
	Let $u\in \widehat{H}_{coll}^q$ be  a minimizer for $\mathcal{M}_h$ such that $\mathcal{M}_h(u)>0$. Then, for $\omega$ given by \eqref{def:omega}, $x(t)=u(\omega t)$ is a classical solution of \eqref{eq:an_kep}-\eqref{eq:energy} in $[0,T/\omega)$ such that
	\begin{itemize}
		 \item $x(0)=q$, $x(T/\omega)=0$;
	     \item $x(t)/|x(t)|\to\vt^*$ as $t\to (T/\omega)^-$, with $\vt^*\in\cerchio$ central configuration for $V$;
	     \item for some positive constant $K$, we have $|x(t)|\sim K(T/\omega-t)^{2/(2+\al)}$ as $t\to(T/\omega)^-$.
	\end{itemize}
\end{theorem}
\begin{proof}
	Since $\mathcal{M}'_h(u)=0$ we have
	\[
	\mathcal{M}'_h(u)[v]=\int_0^T\dot{u}\cdot\dot{v}\int_0^T(h+V(u))+\frac12\int_0^T|\dot{u}|^2\int_0^T\nabla V(u)\cdot v=0,
	\]
	for every $v\in H_0^1([0,T];\R^2)$ and so, since $\mathcal{M}_h(u)>0$
	\[
	\omega^2\int_0^T\dot{u}\cdot\dot{v}+\int_0^T\nabla V(u)\cdot v=0,
	\]
	for every $v\in H_0^1([0,T];\R^2)$. In other words, $u$ is a weak solution of the equation
	\begin{equation}\label{eq:omega_eq}
	\omega^2 \ddot{u}=\nabla V(u),
	\end{equation}
	but also, by standard regularity arguments and by Lemma \ref{lem:archetto}, a classical solution of the same equation in $[0,T)$. Now, it is readily checked that $x(t)=u(\omega t)$ solves \eqref{eq:an_kep}-\eqref{eq:energy} in $[0,T/\omega)$ and that the required boundary conditions are satisfied. The limiting behaviours as $t\to (T/\omega)^-$ follow from the well-known asymptotic estimates (see \cite{BFT2,BTV,BTVplanar}). Moreover, from equation \eqref{eq:omega_eq}, we deduce that there exists $k\in\R$ such that
	\[
	\frac{\omega^2}{2}|\dot{u}(t)|^2-V(u(t))=k,
	\]
	for every $t\in[0,T)$. Integrating the above equation over $[0,T)$, we necessarily get $k=h$ and the energy equation \eqref{eq:energy} for $x$ holds as well.
\end{proof}

\subsection{Existence through direct methods}

In what follows, we show the existence of minimizers for the Maupertuis' functional, which correspond to collision trajectories through Theorem \ref{thm:maupertuis}. However, it will be clear that such motions cannot start too much far from the singularity. The initial distance $r=|q|$ of the particle is indeed linked to the well-known \emph{Lagrange-Jacobi} inequality (see for instance \cite{WintBook}), which we prove below in our setting.

\begin{lemma}[Lagrange-Jacobi inequality]\label{lem:lj}
	Define $U_{min}=\min\limits_{\vt\in\cerchio}U(\vt)$ and 
	\[
	r_{LJ}=\left[\frac{(2-\al)U_{min}}{-2h}\right]^{\frac{1}{\al}}>0.
	\]
	For every solution $x$ of \eqref{eq:an_kep}-\eqref{eq:energy} such that $|x|<r_{LJ}$, we have that the moment of inertia $I(x(t))=\frac12|x(t)|^2$ is strictly convex with respect to $t$. In particular, for a solution $x(t)=r(t)e^{i\vt(t)}$ which collides with the origin after a time $T>0$, we have $r'(t)<0$ in $(0,T)$.
\end{lemma}
\begin{proof}
	By standard calculations and using \eqref{eq:an_kep}-\eqref{eq:energy} we obtain
	\[
	\begin{aligned}
	\frac{d^2}{dt^2}I(x(t))&=\langle\nabla V(x(t)),x(t)\rangle+2(V(x(t))+h) \\
	&=|x(t)|^{-\al}(2-\al)U(\vt(t))+2h \\
	&>(2-\al)r_{LJ}^{-\al}U_{min}+2h=0.
	\end{aligned}
	\]
\end{proof}

The previous result suggests to consider a smaller minimization set than $\widehat{H}_{coll}^q$ and to require the \emph{natural constraint} for a path to do not leave the ball where it started from. To be precise, given $r>0$ and $q\in\partial B_r$, we introduce the set of all the $H^1$-paths which start in $q$ and collapse in the origin in finite time, without leaving the ball $B_r$
\[
H_{coll}^q\uguale\{u\in H^1([0,1];\R^2):\ u(0)=q,\ u(1)=0,\ |u(s)|\leq r\ \mbox{for every}\ s\in[0,1]\}.
\]
Here and later, in order to simplify the notation, we have set $T=1$.

We now present two lemmata which allow us to apply the direct method of the calculus of variations; to this aim, we will often make use of the Poincar\'e inequality, which clearly holds in the space $H_{coll}^q$.

\begin{lemma}\label{lem:bounds}
	For every $q\in B_{r_{LJ}}$, there exists a positive constant $C$ such that
	\[
	0<C\leq\inf\limits_{u\in H_{coll}^q}\mathcal{M}_h(u)<+\infty.
	\]
\end{lemma}
\begin{proof}
	Fix $r\in(0,r_{LJ})$ and $q\in\partial B_r$. From the definition of $r_{LJ}$ given in Lemma \ref{lem:lj}, we have that, for every $u\in H_{coll}^q$
	\begin{equation}\label{eq:bounds_step}
	\int_0^1(h+V(u))\,ds\geq\int_0^1(h+r^{-\al}U_{min})\,ds> h+r_{LJ}^{-\al}U_{min}=-\frac{\al h}{2-\al}>0.
	\end{equation}
	Moreover, for $u\in H_{coll}^q$, we can write
	\[
	r=|u(1)-u(0)|\leq\int_0^1|\dot{u}|\,ds\leq\left(\int_0^1|\dot{u}|^2\,ds\right)^{1/2}
	\]
	and so, together with \eqref{eq:bounds_step}, we obtain
	\[
	\mathcal{M}_h(u)=\int_0^1|\dot{u}|^2\,ds\int_0^1(h+V(u))\,ds\geq-\frac{\al hr^2}{2-\al}=C>0,\quad\mbox{for every}\ u\in H_{coll}^q.
	\]	
	Moreover, since $u\in H_{coll}^q$, then $\dot{u}\in L^2([0,1];\R^2)$ and $V(u)\in L^1([0,1];\R^2)$, by means of the limiting behaviour provided in Theorem \ref{thm:maupertuis}. This proves the upper bound and concludes the proof.
\end{proof}

\begin{lemma}\label{lem:coercivity}
	For every $q\in B_{r_{LJ}}$, $\mathcal{M}_h$ is coercive on $H_{coll}^q$.
\end{lemma}
\begin{proof}
	Fix $r\in(0,r_{LJ})$ and $q\in\partial B_r$ and consider $(u_n)_n\sset H_{coll}^q$, such that $\|u_n\|_{H^1}\to+\infty$ as $n\to+\infty$. Since $|u_n(s)|\leq r$ for every $s\in[0,1]$ and for every $n\in\N$, we obtain that necessarily
	\[
	\lim\limits_{n\to+\infty}\int_0^1|\dot{u}_n|^2\,ds=+\infty
	\]
	and so, together with \eqref{eq:bounds_step}, we conclude that $\mathcal{M}_h(u_n)\to+\infty$.
\end{proof}

We are about to prove that a minimizer of  $\mathcal{M}_h$ exists and thus, invoking Theorem \ref{thm:maupertuis}, a collision trajectory $x(t)$ satisfying \eqref{eq:an_kep}-\eqref{eq:energy} can be provided.

\begin{theorem}\label{thm:coll_orb}
	Given $h<0$ and $r_{LJ}>0$ as in Lemma \ref{lem:lj}, the Maupertuis' functional
	\[   
	\mathcal{M}_h(u)=\int_0^1|\dot{u}|^2\,ds\int_0^1(h+V(u))\,ds
	\]   
	admits at least a minimizer $u\in H_{coll}^q$ at a positive level, for every $q\in B_{r_{LJ}}$. 
\end{theorem}
\begin{proof}
	Let us fix $r\in(0,r_{LJ})$ and $q\in\partial B_r$. Since the weak convergence implies the uniform one in $H^1$, we first observe that $H_{coll}^q$ is weakly closed in $H^1$.
	
	Now, let us consider a sequence $(u_n)_n\sset H_{coll}^q$ such that
	\[
	\mathcal{M}_h(u_n)\ninf\inf\limits_{u\in H_{coll}^q}\mathcal{M}_h(u).
	\]
	From \ref{lem:bounds} and \ref{lem:coercivity} we have that $(u_n)_n$ is bounded in $H^1$ and so $u_n\wconv u$ in $H^1$. In particular, since $H_{coll}^q$ is weakly closed, $u\in H_{coll}^q$. Moreover, from the Poincar\'e inequality, we deduce that 
	\begin{equation}\label{eq:coll_orb_step2}
	\int_0^1|\dot{u}|^2\,ds\leq\liminf\limits_{n\to\infty}\int_0^1|\dot{u}_n|^2\,ds
	\end{equation}
	and, from Lemma \ref{lem:bounds}, for every $n\in\N$ we obtain that 
	\[
	0<C\leq\mathcal{M}_h(u_n)<\infty
	\]
	and thus $V(u_n)\in L^1(0,1)$, for every $n\in\N$. This implies that the set $\{t\in[0,1]:\ u_n(t)=0\}$ has null measure and hence, since $u_n$ converges to $u$ uniformly, we have that $V(u_n)\to V(u)$ almost everywhere. We can now use Fatou's Lemma to deduce that $V(u)\in L^1(0,1)$ and that
	\[
	\int_0^1(h+V(u))\,ds\leq\liminf\limits_{n\to\infty}\int_0^1(h+V(u_n))\,ds.
	\]
	This, together with \eqref{eq:coll_orb_step2}, proves that 
	\[
	\mathcal{M}_h(u)\leq\liminf\limits_{n\to+\infty}\mathcal{M}_h(u_n)=\inf\limits_{u\in H_{coll}^q}\mathcal{M}_h(u).\qedhere
	\]
\end{proof}

\subsection{A compactness lemma}

Theorem \ref{thm:coll_orb} shows that once $h<0$ and $q\in B_{r_{LJ}}$ are fixed, we can always find at least a minimizer of the Maupertuis' functional in the space $H_{coll}^q$. In this way, if we fix $r\in(0,r_{LJ})$, we can define a function $\psi_h\colon\partial B_r\to\R^+$ such that
\[
\psi_h(q)\uguale\min\limits_{u\in H_{coll}^q}\mathcal{M}_h(u)\quad\mbox{for}\ q\in\partial B_r.
\] 
\begin{remark}\label{rem:hill}
	In the next proposition we are going to make use of the Jacobi-length functional
	\[
	\mathcal{L}_h(u)=\int_0^1|\dot{u}|\sqrt{h+V(u)}\,ds,
	\]
	which, for a path $u\in H^1([0,1];\R^2)$, is well-defined if and only if $|\dot{u}|>0$ and $h+V(u)>0$. Therefore, it makes sense to consider paths which live far from the boundary of the Hill's region $\partial\mathcal{R}_h$. Nonetheless, with our choice of $r_{LJ}$ provided in Lemma \ref{lem:lj}, this condition is already satisfied. Indeed, taking $q\in B_{r_{LJ}}$ and $u\in H_{coll}^q$, we can write
	\[
	h+V(u)\geq h+r_{LJ}^{-\al}U_{min}=-\frac{\al h}{2-\al}=C>0.
	\]
	Moreover, from the energy equation, we clearly have $|\dot{u}|>0$.
\end{remark}

\begin{proposition}\label{prop:jacobi_lip}
	For $h<0$ and $r\in(0,r_{LJ})$, the function $\psi_h$ is Lipschitz continuous on $\partial B_r$. In other words, there exists $L=L(r_{LJ})>0$ such that
	\[
	|\psi_h(q_2)-\psi_h(q_1)|\leq L|\vt_2-\vt_1|,\qquad\mbox{for every}\ q_1=re^{i\vt_1},q_2=re^{i\vt_2}\in\partial B_r.
	\]
\end{proposition}
\begin{proof}
	Fix $h<0$ and $r\in(0,r_{LJ})$. Given $q\in\partial B_r$, for a path $u\in H_{coll}^q$ we can define the Jacobi-length functional
	\[
	\mathcal{L}_h(u)=\int_0^1|\dot{u}|\sqrt{h+V(u)}\,dt,
	\] 		
	which is linked to $\mathcal{M}_h$ in this way:
	\[
	2\min\limits_{H_{coll}^q}\mathcal{M}_h=(\min\limits_{H_{coll}^q}\mathcal{L}_h)^2.
	\]
	Therefore, if we define the function $\omega_h(q)\uguale\min\limits_{H_{coll}^q}\mathcal{L}_h$ for $q\in\partial B_r$ and we show that it is Lipschitz continuous we are done.
	
	Fix $q_1=re^{i\vt_1},q_2=re^{i\vt_2}\in\partial B_r$ and consider the circular path
	\[
	u_{arc}(t)=re^{i((1-t)\vt_1+t\vt_2)},\quad\mbox{for}\ t\in[0,1].
	\]
	We have
	\[
	\begin{aligned}
	\mathcal{L}_h(u_{arc})&=r|\vt_2-\vt_1|\int_0^1\sqrt{h+r^{-\al}U(\vt(t))}\,dt \\
	&<L|\vt_2-\vt_1|,\quad\text{where}\ L=L(r_{LJ})=r_{LJ}^{1-\al/2}\sqrt{U_{max}}.
	\end{aligned}
	\]
	Now, since $\mathcal{L}_h$ is a length, it is invariant under time rescaling and so we can write
	\[
	\min_{H_{coll}^{q_1}}\mathcal{L}_h\leq\mathcal{L}_h(u_{arc})+\min_{H_{coll}^{q_2}}\mathcal{L}_h.
	\]		
	Finally, from the definition of $\omega_h$, we obtain
	\[
	\omega_h(q_1)\leq \omega_h(q_2)+L|\vt_2-\vt_1|
	\]
	and, with the same argument
	\[
	\omega_h(q_2)\leq \omega_h(q_1)+L|\vt_2-\vt_1|. \qedhere
	\]
\end{proof}

\begin{corollary}\label{cor:lip}
	Given $h<0$, $r^*\in(0,r_{LJ})$, $h^*\in(h,0)$ and $\vt^*\in\cerchio$, consider three sequences $(h_k)_k\sset\R$, $(r_k)_k\sset\R^+$ and $(\vt_k)_k\sset\cerchio$ such that
	\begin{itemize}
		\item $h_k\in(h,0)$ for every $k\in\N$ and $h_k\to h^*$ as $k\to+\infty$\emph{;}
		\item $0<r_k<r_{LJ}$ for every $k\in\N$ and $r_k\to r^*$ as $k\to+\infty$\emph{;}
		\item $q_k=r_ke^{i\vt_k}\to q^*=r^*e^{i\vt^*}$ as $k\to+\infty$.
	\end{itemize}
	Then
	\[
	\min\limits_{H_{coll}^{q_k}}\mathcal{M}_{h_k}\leq\min\limits_{H_{coll}^{q^*}}\mathcal{M}_{h^*}+O(|\vt^*-\vt_k|)+O(|h^*-h_k|)+O(|r^*-r_k|),
	\]
	as $k\to+\infty$.
\end{corollary}
\begin{proof}
	Fix $h<0$, $r^*\in(0,r_{LJ})$, $h^*\in(h,0)$ and $\vt^*\in\cerchio$. Consider the three sequences as in the statement and fix $k\in\N$. We can write 
	\[
	\psi_{h_k}(q_k)=\min\limits_{H_{coll}^{q_k}}\mathcal{M}_{h_k},\quad \psi_{h^*}(q^*)=\min\limits_{H_{coll}^{q^*}}\mathcal{M}_{h^*},
	\]
	so that
	\[
	\psi_{h_k}(q_k)-\psi_{h^*}(q^*)=\psi_{h_k}(q_k)-\psi_{h_k}(q^*)+\psi_{h_k}(q^*)-\psi_{h^*}(q^*).
	\]
	Let us start by the estimate of the term $\psi_{h_k}(q^*)-\psi_{h^*}(q^*)$ on the right-hand side. Consider $u^*\in H_{coll}^{q^*}$ such that
	\[
	\psi_{h}(q^*)=\min_{H_{coll}^{q^*}}\mathcal{M}_{h^*}=\mathcal{M}_{h^*}(u^*).
	\]
	For the minimality of $u^*$, we obtain
	\[
	\begin{aligned}
	\psi_{h_k}(q^*)-\psi_{h^*}(q^*)&\leq\mathcal{M}_{h_k}(u^*)-\mathcal{M}_{h^*}(u^*) \\
	&\leq|h^*-h_k|\int_0^1|\dot{u}^*|^2\,ds\leq C_1|h^*-h_k|,
	\end{aligned}
	\]
	with $C_1>0$. Now, concerning the term $\psi_{h_k}(q_k)-\psi_{h_k}(q^*)$, if we consider the point
	\[
	q_k'=\frac{r^*}{r_k}q_k\in\partial B_{r^*}
	\]
	we can write
	\[
	\psi_{h_k}(q_k)-\psi_{h_k}(q^*)=\psi_{h_k}(q_k)-\psi_{h_k}(q_k')+\psi_{h_k}(q_k')-\psi_{h_k}(q^*).
	\]
	Take $v_k\in H_{coll}^{q_k'}$ such that
	\[
	\psi_{h_k}(q_k')=\min_{H_{coll}^{q_k'}}\mathcal{M}_{h_k}=\mathcal{M}_{h_k}(v_k)
	\]
	and define the path
	\[
	\tilde{v}_k=\frac{r_k}{r^*}v_k\in H_{coll}^{q_k}.
	\]    	
	We can write
	\[
	\begin{aligned}
	\psi_{h_k}(q_k)-\psi_{h_k}(q_k')&\leq\mathcal{M}_{h_k}(\tilde{v}_k)-\mathcal{M}_{h_k}(v_k) \\
	&=\int_0^1|\dot{\tilde{v}}_k|^2\,ds\int_0^1(h_k+V(\tilde{v}_k))\,ds-
	\int_0^1|\dot{v}_k|^2\,ds\int_0^1(h_k+V(v_k))\,ds \\
	&=\int_0^1|\dot{v}_k|^2\,ds\int_0^1\left[h_k\left(\frac{r_k}{r^*}\right)^2-h_k+\left(\frac{r_k}{r^*}\right)^{2-\al}V(v_k)-V(v_k)\right]\,ds \\
	&=\int_0^1|\dot{v}_k|^2\,ds\int_0^1\left[h_k\frac{r_k+r^*}{(r^*)^2}(r_k-r^*)+V(v_k)\frac{r_k^{2-\al}-(r^*)^{2-\al}}{(r^*)^{2-\al}}\right]\,ds \\
	&\leq C_2(|r^*-r_k|),
	\end{aligned}
	\]
	with $C_2>0$. Finally, since $q_k',q^*\in\partial B_{r^*}$, we can apply Proposition \ref{prop:jacobi_lip} to obtain
	\[
	\psi_{h_k}(q_k')-\psi_{h_k}(q^*)\leq L|\vt^*-\vt_k|. \qedhere
	\] 
\end{proof}

Now we prove the following compactness lemma on sequences of minimizers of the Maupertuis' functional.

\begin{lemma}\label{lem:min_comp}
	Given $h<0$, $r^*\in(0,r_{LJ})$, $h^*\in(h,0)$ and $\vt^*\in\cerchio$, consider three sequences $(h_k)_k\sset\R^+$, $(r_k)_k\sset\R$ and $(\vt_k)_k\sset\cerchio$ such that
	\begin{itemize}
		\item $h_k\in(h,0)$ for every $k\in\N$ and $h_k\to h^*$ as $k\to+\infty$\emph{;}
		\item $0<r_k<r_{LJ}$ for every $k\in\N$ and $r_k\to r^*$ as $k\to+\infty$\emph{;}
		\item $q_k=r_ke^{i\vt_k}\to q^*=r^*e^{i\vt^*}$ as $k\to+\infty$.
	\end{itemize} 
	With a slight abuse of notation, define the classes
	\[
	H_{coll}^k=\{u\in H^1([0,1];\R^2):\ u(0)=q_k,\ u(1)=0,\ |u(s)|\leq r_k\ \mbox{for every}\ s\in[0,1]\}
	\]
	and
	\[
	H_{coll}^*=\{u\in H^1([0,1];\R^2):\ u(0)=q^*,\ u(1)=0,\ |u(s)|\leq r^*\ \mbox{for every}\ s\in[0,1]\}.
	\]
	If $u_k$ is a minimizer of $\mathcal{M}_{h_k}$ in $H_{coll}^k$ for every $k\in\N$, then
	\begin{itemize}
		\item[$(i)$] $u_k\to u^*$ in $H^1([0,1];\R^2)$;
		\item[$(ii)$] $u_k\to u^*$ in $C^2([0,b];\R^2)$, for every $b<1$.
	\end{itemize}
	In particular, $u^*$ is a minimizer of $\mathcal{M}_{h^*}$ in the class of paths $H_{coll}^*$.		
\end{lemma}
\begin{proof}
	Fix $h<0$, $r^*\in(0,r_{LJ})$, $h^*\in(h,0)$ and $\vt^*\in\cerchio$ and consider the sequences $(h_k)_k$, $(r_k)_k$ and $(\vt_k)_k$ as in the statement. For every $k\in\N$, consider a minimizer $u_k$ of $\mathcal{M}_{h_k}$ in $H_{coll}^k$. For $k\in\N$ and for every $s\in[0,1]$, following Remark \ref{rem:hill}, we have 
	\[
	h_k+V(u_k(s))>h+r_k^{-\al}U_{min}>h+r_{LJ}^{-\al}U_{min}=C>0.
	\]
	In this way, we can write
	\[
	\inf\limits_{H_{coll}^k}\mathcal{M}_{h_k}=\mathcal{M}_{h_k}(u_k)=\int_0^1|\dot{u}_k|^2\,ds\int_0^1(h_k+V(u_k))\,ds>C\int_0^1|\dot{u}_k|^2\,ds
	\]
	and so, by Lemma \ref{lem:bounds} and the Poincar\'e inequality, the sequence $(u_k)_k$ is bounded in $H^1$ and hence $u_k\wconv u^*$ in $H^1$ and uniformly. So, from Fatou's lemma, we have
	\[
	\mathcal{M}_{h^*}(u^*)\leq\liminf\limits_{k\to\infty}\mathcal{M}_{h_k}(u_k).
	\]
	Now, suppose by contradiction that there exists a minimizer $u_{min}$ of $\mathcal{M}_{h^*}$ in $H_{coll}^*$ such that 
	\[
	\mathcal{M}_{h^*}(u_{min})<\mathcal{M}_{h^*}(u^*).
	\]
	From Corollary \ref{cor:lip}, we actually obtain that, as $k\to+\infty$
	\[
	\mathcal{M}_{h_k}(u_k)\leq\mathcal{M}_{h^*}(u_{min})+O(|\vt^*-\vt_k|)+O(|h^*-h_k|)+O(|r^*-r_k|)
	\]
	and so
	\[
	\liminf\limits_{k\to\infty}\mathcal{M}_{h_k}(u_k)\leq\mathcal{M}_{h^*}(u_{min}),
	\]
	which leads to a contradiction. Therefore, $u^*$ is a minimizer. Moreover, the same argument leads to a strong convergence in $H^1$ assuming by contradiction that
	\[
	\int_0^1|\dot{u}^*|^2\,ds<\limsup\limits_{k\to\infty}\int_0^1|\dot{u}_k|^2\,ds.
	\] 
	
	Finally, since $r_k\to r^*>0$, we have $\inf_k r_k>0$ and so we can consider $B_{\tilde{r}}$ with $\tilde{r}=\frac12\inf_k r_k$. From Lemma \ref{lem:lj}, there exists a sequence $(b_k)_k$ such that $|u_k(b_k)|=\tilde{r}$ and $0<b_k<1$ for every $k\in\N$. Defining $b=\inf_kb_k$, again from Lemma \ref{lem:lj} we deduce that $0<b<1$. In this way, we obtain
	\[
	\omega^2\ddot{u}_k(t)=\nabla V(u_k(t))\quad\mbox{for}\ t\in[0,b],\ \text{for every}\ k\in\N
	\]
	and thus $\nabla V(u_k)$ converges uniformly to $\nabla V(u^*)$ on $[0,b]$. This proves that $u_k$ converges in $C^2([0,b])$.
\end{proof}

\section{Proof of the main theorem}\label{sec:main}

In general, the minimizer provided by Theorem \ref{thm:coll_orb} is not unique. However, for a particular class of collision trajectories, we have the following result.

\begin{lemma}\label{lem:un_hom}
	Assume \eqref{hyp:vt_star}-\eqref{hyp:V1}. Given $h<0$, $r^*\in(0,[-U(\vt^*)/h]^{1/\al}]$ and taking $q^*=r^*e^{i\vt^*}$, there exists a unique minimizer for the Maupertuis' functional $\mathcal{M}_h$ in $H_{coll}^{q^*}$. This arc is nothing but the monotone $\vt^*$-homothetic collision trajectory.
\end{lemma}
\begin{proof}
	Let us define the homothetic trajectory as $u_{hom}(s)=r_{hom}(s)e^{i\vt^*}$, with $r_{hom}(0)=r^*$, $r_{hom}(1)=0$ and $\dot{r}_{hom}(s)<0$ for every $s\in(0,1)$. For every $u\in H_{coll}^{q^*}$ we can write $u(s)=r(s)e^{i\vt(s)}$, so that
	\[
	\begin{aligned}
	\mathcal{M}_h(u) &=\int_0^1|\dot{r}e^{i\vt}+ir\dot{\vt}e^{i\vt}|^2\,ds\int_0^1(h+r^{-\al}U(\vt))\,ds \\
	&\geq\int_0^1\dot{r}^2\,ds\int_0^1(h+r^{-\al}U_{min})\,ds\geq\mathcal{M}_h(u_{hom}).
	\end{aligned}
	\]	
	Indeed, the last inequality is strict if $r(s)$ is not monotone in $(0,1)$, otherwise we end up with an equality, since the Maupertuis' functional is invariant under time rescaling (see the proof of Theorem \ref{thm:maupertuis}).
\end{proof}

In order to enlarge the set of those minimizers which are also unique, we are going to exploit the dynamical features of our problem. Therefore, we come back to our study started in Sections \ref{sec:coll_man}-\ref{sec:st_man}. From Lemma \ref{lem:stable_manif} and Remark \ref{rem:top_conj} we have a precise characterization of the local stable manifold $\mathcal{W}_{loc}^s$ of the fixed point $(0,\vt^*,\vt^*+\pi)$. Given a starting point $q$, the Maupertuis' functional $\mathcal{M}_h$ does not necessarily admit a unique minimizer on the class of collision paths $H_{coll}^q$ but, actually, this set of minimizers is in 1-1 correspondence with the set of their starting velocities. This fact is a consequence of the uniqueness of the solutions of the relative Cauchy problems and it suggests to establish a link between every minimizer and its initial velocity. In this way, we can introduce the set of the reparameterizations through McGehee's coordinates of every minimizer with starting position $q$ as follows
\begin{equation}\label{defn:m_h}
\mathfrak{m}_h(q)=\{\gamma_\vp:\ \mbox{rep}(\gamma_\vp)\ \mbox{minimizes}\ \mathcal{M}_h\ \mbox{in}\ H_{coll}^q,\ \mbox{with }\ \vp\in\cerchio\},
\end{equation}
with $\gamma_\vp=\gamma_\vp(\tau)$ for $\tau>0$. With this notation, the angle $\vp$ is nothing but the direction of the starting velocity, since its module is already fixed for the conservation of energy.

Following this preliminary discussion, we state and prove below our main result, which is nothing but the planar unperturbed version of the main theorem presented in the Introduction of this paper, in a negative energy shell.

\begin{theorem}\label{thm:main}
	Assume \eqref{hyp:vt_star}-\eqref{hyp:V1}. Given $h<0$, there exist $\bar{r}>0$ and $\bar\delta>0$ such that, defining 
	\[
	Q_0=\{q_0=re^{i\vt_0}:\ r\in(0,\bar{r})\ \mbox{and}\ \vt_0\in(\vt^*-\bar\delta,\vt^*+\bar\delta)\},
	\]
	and $\mathfrak{m}_h$ as in \eqref{defn:m_h}, then
	\begin{equation}\label{eq:main}
	\mathcal{W}_{loc}^s(\bar r,\bar\delta)=\bigcup\limits_{q_0\in Q_0}\mathfrak{m}_h(q_0).
	\end{equation}
	Moreover, for every $q_0\in Q_0$, the Maupertuis' functional $\mathcal{M}_h$ admits a unique minimizer in $H_{coll}^{q_0}$.
\end{theorem}
\begin{proof}
    Fix $h<0$ and, adopting the notations of Lemma \ref{lem:stable_manif} and Lemma \ref{lem:lj}, choose 
    \[
    \bar{r}=\min\{r_{loc},r_{LJ}\}.
    \]
    
    We start with showing that the uniqueness of a minimizer follows from inclusion ($\supseteq$) in \eqref{eq:main}. Indeed, Lemma \ref{lem:stable_manif} guarantees that $\mathcal{W}_{loc}^s$ is the graph of a $\mathcal{C}^2$-function, which associates to every initial position $q_0\in Q_0$ a unique initial velocity. Therefore, once $q_0$ is fixed, there necessarily exists a unique minimal arc solving the correspondent fixed-end collision problem.
	
	\emph{Inclusion} ($\supseteq$):  Take $q_0=re^{i\vt_0}$, with $r\in(0,\bar{r})$ and $\vt_0\in(\vt^*-\delta_{loc},\vt^*+\delta_{loc})$ (as in Lemma \ref{lem:stable_manif}). Therefore, by Theorem \ref{thm:coll_orb}, there exists $\gamma=\gamma_{\vp_0}\in\mathfrak{m}_h(q_0)$ for some $\vp_0\in\cerchio$. Our goal is to show that, up to make $\delta_{loc}$ smaller, $\gamma$ is entirely contained in the local stable manifold $\mathcal{W}_{loc}^s$. The orbit $\gamma(\tau)=(r(\tau),\vt(\tau),\vp(\tau))$ solves \eqref{eq:mcgehee} in $(0,+\infty)$ so that, by the definition of stable manifold, we have
	\begin{equation}\label{thm:st_man}
	\gamma\in\mathcal{W}^s\iff\begin{cases}
	\begin{aligned}
	&r(\tau)\to 0\quad&\mbox{as}\ \tau\to+\infty \\
	&\vt(\tau)\to\vt^*\quad&\mbox{as}\ \tau\to+\infty \\
	&\vp(\tau)\to \vp^*=\vt^*+\pi\quad&\mbox{as}\ \tau\to+\infty,
	\end{aligned}
	\end{cases}
	\end{equation}
	but also, from Lemma \ref{lem:stable_manif}
	\[
	\gamma\in\mathcal{W}_{loc}^s\iff\ \text{for every}\ \tau>0,\  \vp(\tau)=\Psi(r(\tau),\vt(\tau)).
	\]
	Assume by contradiction that there exists $(\vt_k(0))_k\sset\cerchio$ such that
	\begin{equation}\label{thm:hyp}
	q_k=re^{i\vt_k(0)}\to q^*=re^{i\vt^*}\quad\mbox{as}\ k\to+\infty,
	\end{equation}
	but there exists a sequence of reparameterized minimizers $(\gamma_k)_k\sset(\mathfrak{m}_h(q_k))_k$ such that
	\[
	\gamma_k=\{\gamma_k(\tau)=(r_k(\tau),\vt_k(\tau),\vp_k(\tau)):\ \tau\geq 0\}\not\sset\mathcal{W}_{loc}^s,\quad\text{for every}\ k\in\N.
	\]		
	Notice that, from Lemma \ref{lem:min_comp} and Lemma \ref{lem:un_hom}, the sequence $(\gamma_k)_k$ converges in $H^1([0,1];\R^2)$, and thus uniformly in $[0,1]$, to the $\vt^*$-homothetic motion. For this reason, we can split our proof in two situations, whose discriminant is the uniform convergence of the sequence $(\vt_k)_k=(\vt_k(\tau))_k$. Indeed, despite the convergence of $(\gamma_k)_k$, for instance it could happen that the sequence of angle functions $(\vt_k)_k$ starts to oscillate dramatically when $k$ goes to $+\infty$.
	
	\emph{Case} 1: Assume that $\gamma_k\not\sset\mathcal{W}_{loc}^s$ for infinite $k$ and that	
	\[
	\lim\limits_{k\to+\infty}\sup\limits_{\tau\geq 0}|\vt_k(\tau)-\vt^*|=0.
	\]
	By \eqref{thm:st_man}, we necessarily have that there exists $\bar{\ve}>0$ such that
	\[
	\limsup\limits_{k\to+\infty}\sup\limits_{\tau\geq 0}|\vp_k(\tau)-\vp^*|=\bar{\ve}>0.
	\]
	In this way, up to subsequences, we can find a sequence $(\tau_k)_k\subseteq[0,+\infty)$ such that 
	\[
	|\vp_k(\tau_k)-\vp^*|=\bar\ve,\quad\mbox{for every}\ k\in\N.
	\]
	Now, we perform the following change of variables and time shifting
	\[
	\begin{cases}
	\tilde{r}_k(\tau)=\frac{r}{r_k(\tau_k)}r_k(\tau+\tau_k) \\
	\tilde{\vt}_k(\tau)=\vt_k(\tau+\tau_k) \\
	\tilde{\vp}_k(\tau)=\vp_k(\tau+\tau_k)
	\end{cases}
	\]
	and, if we define $\la_k=r_k(\tau_k)/r\leq 1$, we have that the orbit $\tilde{\gamma}_k(\tau)=(\tilde{r}_k(\tau),\tilde{\vt}_k(\tau),\tilde{\vp}_k(\tau))_{\tau\geq 0}$ solves the system
	\[
	\begin{cases}
	\tilde{r}_k'=2\tilde{r}_k(U(\tilde{\vt}_k)+h_k\tilde{r}_k^\al)\cos(\tilde{\vp}_k-\tilde{\vt}_k) \\
	\tilde{\vt}_k'=2(U(\tilde{\vt}_k)+h_k\tilde{r}_k^\al)\sin(\tilde{\vp}_k-\tilde{\vt}_k) \\
	\tilde{\vp}_k'=U'(\tilde{\vt}_k)\cos(\tilde{\vp}_k-\tilde{\vt}_k)+\al U(\tilde{\vt}_k)\sin(\tilde{\vp}_k-\tilde{\vt}_k),
	\end{cases}
	\]
	where $h_k=\la_k^\al h$ and so $h_k\in[h,0)$. Now, denoting by $x_k$ the reparameterization of the trajectory $\gamma_k$ in time $t$ in the configurations space, $x_k$ solves the problem
	\[
	\begin{cases}
	\ddot{x}_k=\nabla V(x_k) \\
	\frac12 |\dot{x}_k|^2-V(x_k)=h,
	\end{cases}
	\]
	for every $k\in\N$. Therefore, the function
	\[
	\tilde{x}_k(t)=\frac{x_k(t_k+\la_k^{1+\al/2}t)}{\la_k},
	\]
	with $t_k$ such that $\tilde{x}_k(0)=x_k(0)$, will solve the problem
	\[
	\begin{cases}
	\frac{d^2}{dt^2}{\tilde{x}_k(t)}=\nabla V(\tilde{x}_k(t)) \\
	\frac12|\frac{d}{dt}\tilde{x}_k(t)|^2-V(\tilde{x}_k(t))=h_k,
	\end{cases}
	\]
	for every $k\in\N$. Hence, under a suitable change of scales, there exists a sequence $(\tilde{u}_k)_k$ of minimizers of the Maupertuis' functional, with starting point respectively in $(\tilde q_k)_k$ such that
	\[
	\tilde{q}_k=\tilde{r}_k(0)e^{i\tilde\vt_k(0)}=re^{i\vt_k(\tau_k)}\to q^*=re^{i\vt^*}\quad\mbox{as}\ k\to+\infty.
	\]
	For Lemma \ref{lem:min_comp}, we have that such a sequence of minimizers converges in $H^1$ and thus uniformly to a minimal arc connecting $q^*$ to the origin. On the other hand, we have that $\tilde{\vp}_k(0)\not\to\vp^*=\vt^*+\pi$ as $k\to+\infty$. This means that the limit arc cannot be the $\vt^*$-homothetic trajectory, which is impossible for Proposition \ref{lem:un_hom}.
	
	\emph{Case} 2: Assume that $\gamma_k\not\in\mathcal{W}^s$ for infinite $k$ and that there exists $\bar\ve>0$ such that
	\begin{equation}\label{thm:hyp_abs}
	\limsup\limits_{k\to+\infty}\sup\limits_{\tau\geq 0}|\vt_k(\tau)-\vt^*|=\bar\ve.
	\end{equation}
	Hence, up to subsequences, there exists a sequence $(\tau_k)_k\sset[0,\infty)$ such that
	\begin{equation}\label{thm:hyp_abs2}
	|\vt_k(\tau_k)-\vt^*|=\bar{\ve},\quad\mbox{for every}\ k\in\N.
	\end{equation}
	Moreover, since from Lemma \ref{lem:min_comp} and Proposition \ref{lem:un_hom} the sequence of minimal collision arcs $(\gamma_k)_k$ converges $C^2$ to the homothetic motion on every bounded interval, we necessarily deduce that $\tau_k\to+\infty$ as $k\to+\infty$. In particular, since every $\gamma_k$ is a collision arc, again from Lemma \ref{lem:min_comp} we have that
	\begin{equation}\label{thm:step}
	r_k(\tau_k)\to 0,\quad\mbox{as}\ k\to+\infty.
	\end{equation}
	
	Now, for every $k\in\N$ define the orbit $\tilde{\gamma}_k=(\tilde{r}_k,\tilde{\vt}_k,\tilde{\vp}_k)$ such that
	\[
	\begin{cases}
	\tilde{r}_k(\tau)=r_k(\tau+\tau_k) \\
	\tilde{\vt}_k(\tau)=\vt_k(\tau+\tau_k) \\
	\tilde{\vp}_k(\tau)=\vp_k(\tau+\tau_k),
	\end{cases}
	\]
	for $\tau\in[-\tau_k,+\infty)$. We have that, for every $k\in\N$, $\tilde{\gamma}_k$ verifies equations
	\[
	\begin{cases}
	\tilde{r}_k'=2\tilde{r}_k(U(\tilde{\vt}_k)+h\tilde{r}_k^\al)\cos(\tilde{\vp}_k-\tilde{\vt}_k) \\
	\tilde{\vt}_k'=2(U(\tilde{\vt}_k)+h\tilde{r}_k^\al)\sin(\tilde{\vp}_k-\tilde{\vt}_k) \\
	\tilde{\vp}_k'=U'(\tilde{\vt}_k)\cos(\tilde{\vp}_k-\tilde{\vt}_k)+\al U(\tilde{\vt}_k)\sin(\tilde{\vp}_k-\tilde{\vt}_k).
	\end{cases}
	\]
	Since $\tau_k\to +\infty$, we have that, for every $T>0$, there exists $\bar{k}\in\N$ such that, for every $k\geq\bar{k}$
	\[
	\tau_k>T.
	\]
	Let us fix $T>0$. For $\tau\in[-T,T]$ and for every $k\geq\bar{k}$ we have
	\[
	\tau+\tau_k\in[0,+\infty)
	\]
	and so
	\[
	\tilde{r}_k(\tau)=r_k(\tau+\tau_k)\leq r_k(0)<\bar{r}.
	\]
	Moreover,
	\[
	\tilde{\vt}_k(\tau),\tilde{\vp}_k(\tau)\in\cerchio
	\] 
	and
	\[
	|\tilde{r}'_k(\tau)|\leq 2\bar{r}(U_{max}-h\bar{r}^\al)=C<+\infty
	\]
	and, with analogous calculations, the same holds for $\tilde{\vt}'_k(\tau)$ and $\tilde{\vp}'_k(\tau)$. From the Ascoli-Arzel\`a theorem, we have that 
	\[
	(\tilde{r}_k,\tilde{\vt}_k,\tilde{\vp}_k)\to(\tilde{r},\tilde{\vt},\tilde{\vp})\quad\mbox{as}\ k\to+\infty
	\]
	uniformly on $[-T,T]$. Moreover, from \eqref{thm:step} we deduce that
	\[
	-\tilde{r}_k^\al(\tau)h\leq- \tilde{r}_k^{\al}(0)h=-r_k^\al(\tau_k)h\to 0\quad\mbox{as}\ k\to+\infty.
	\] 
	This, together with the uniform convergence, implies that $(\tilde{r},\tilde{\vt},\tilde{\vp})$ satisfy the equations
	\begin{equation}\label{thm:eq_limit}
	\begin{cases}
	\tilde{r}'=2\tilde{r}U(\tilde{\vt})\cos(\tilde{\vp}-\tilde{\vt}) \\
	\tilde{\vt}'=2U(\tilde{\vt})\sin(\tilde{\vp}-\tilde{\vt}) \\
	\tilde{\vp}'=U'(\tilde{\vt})\cos(\tilde{\vp}-\tilde{\vt})+\al U(\tilde{\vt})\sin(\tilde{\vp}-\tilde{\vt}),
	\end{cases}
	\end{equation}
	on $[-T,T]$. Repeating the same argument for every $T>0$, we have that the sequences converges uniformly on every compact of $\R$, with limit defined  and verifying \eqref{thm:eq_limit} on the whole $\R$. Moreover, let us notice that
	\[
	\tilde{r}_k(0)=r_k(\tau_k)\to 0=\tilde{r}(0)\quad\mbox{as}\ k\to+\infty
	\]
	and so, for the uniqueness of the solution of a Cauchy problem, we actually deduce that $\tilde{r}(\tau)\equiv0$. This means that the solution of \eqref{thm:eq_limit} is actually a motion on the collision manifold $\{r=0\}$.
	
	Let us now investigate the asymptotic behaviour of $\tilde \vt$.  From the non-degeneracy of $\vt^*$, it is not restrictive to assume that $\bar{\ve}=d/2$ in \eqref{thm:hyp_abs}, where
	\[
	d\uguale\min\{|\vt^*-\hat{\vt}|:\,\hat{\vt}\in\cerchio,\ U'(\hat{\vt})=0,\ \hat{\vt}\neq\vt^*\}.
	\]
	In this way, from \eqref{thm:hyp_abs2}, we can deduce that for every $k\in\N$
	\[
	|\vt_k(\tau)-\vt^*|<\frac d2,\quad\mbox{for every}\ \tau\in[0,\tau_k)
	\]
	and so, in other words, for every $k\in\N$
	\[
	|\tilde{\vt}_k(\tau)-\vt^*|<\frac d2,\quad\mbox{for every}\ \tau\in[-\tau_k,0).
	\]
	From the convergence of $\tilde\vt_k$ to $\tilde{\vt}$ we can easily deduce that
	\begin{equation}\label{thm:eq1}
	|\tilde{\vt}(\tau)-\vt^*|\leq \frac d2,\quad\mbox{for every}\ \tau\in(-\infty,0).
	\end{equation}
	Now, it is known (see \cite{DevProgMath1981,BTVplanar}) that
	\[
	\lim\limits_{\tau\to\pm\infty}\tilde{\vt}(\tau)=\vt^\pm,\quad\lim_{\tau\to\pm\infty}\tilde{\vp}(\tau)=\vt^\pm+\pi,
	\]
	with $\vt^\pm$ central configuration for $U$. Assume by contradiction that $\vt^-\neq\vt^*$, i.e. that for every $\ve$ there exists $\tau_\ve\in\R$ such that 
	\[
	\tau<\tau_\ve\implies |\tilde{\vt}(\tau)-\vt^-|<\ve.
	\]
	Choosing $\ve=d/4$ and using \eqref{thm:eq1} we head to a contradiction and so necessarily $\vt^-=\vt^*$.
	
	To conclude the proof, consider the function
	\[
	v(\tau)=\sqrt{U(\tilde{\vt}(\tau))}\cos(\tilde{\vp}(\tau)-\tilde{\vt}(\tau)),
	\]
	which is non-decreasing and non-constant on every solution of \eqref{thm:eq_limit} (see Theorem 4, \cite{BTVplanar}), so that
	\[
	-\sqrt{U_{min}}=-\sqrt{U(\vt^*)}=\lim\limits_{\tau\to-\infty} v(\tau)<\lim\limits_{\tau\to+\infty}v(\tau)=-\sqrt{U(\vt^+)}.
	\]
	This is clearly a contradiction since $\vt^*$ is a global minimal central configuration for $U$.    
	
	\emph{Inclusion} $(\sset)$: From any $(r,\vt_0,\vp_0)\in\mathcal{W}_{loc}^s$ it starts a unique orbit that ends in the equilibrium point $(0,\vt^*,\vt^*+\pi)$, which connects the point $q_0=re^{i\vt_0}$ to the origin in the configuration space. This is nothing but a reparameterization of a minimal fixed-end arc: indeed, a minimizer from $q_0$ to the origin exists by means of Theorem \ref{thm:coll_orb} and it is unique, as we have already shown.    
\end{proof}

\section{General setting}\label{sec:generalizations}

This final section collects some useful remarks for the adaptation of the previous proof in the general setting presented in the introduction. This material is mainly thought to ease the reader's comprehension, for it will be clear that the argument used in the proof is exactly the same. We made the choice to split the generalization in two sub-cases. In the first one, we take into account a perturbed potential and a conservative system with possibly non-negative energy. The second one is focused on the higher dimensional case.

\subsection{$\boldsymbol{d=2}$, $\boldsymbol{W\not\equiv 0}$, $\boldsymbol{h\in\R}$}

As a first step, we set equation \eqref{eq:an_kep} again in the plane ($d=2$), but we perturb our potential $V$ exactly as stated in \eqref{hyp:V}. Moreover, we wish to work also in non-negative energy shells, so that equation \eqref{eq:energy} will be given with $h\in\R$. Using polar coordinates $(r,\vt)$ and adopting the same argument as in Section \ref{sec:coll_man}, we find an analogous of system \eqref{eq:mcgehee} in our actual setting, i.e., the dynamical system
\begin{equation}\label{eq:mcgehee_w}
\begin{cases}
r' = 2r(U(\vt)+r^\al W(r,\vt)+hr^\al) \cos (\vp-\vt) \\
\vt' = 2(U(\vt)+r^\al W(r,\vt)+hr^\al) \sin (\vp-\vt) \\
\vp' =(U'(\vt)-r^\al W_\vt)\cos(\vp-\vt) +(\al U(\vt)+r^{\al+1}W_r) \sin(\vp-\vt)
\end{cases},
\end{equation}
where we have set
\[
W_r=\frac{\partial W}{\partial r}(r,\vt),\quad W_\vt=\frac{\partial W}{\partial\vt}(r,\vt).
\]
It is easy to notice that the collision manifold $\{r=0\}$ induced by \eqref{eq:mcgehee_w} is nothing but the one described by \eqref{eq:coll_man} in Section \ref{sec:coll_man}, regardless of the sign of $h$. This fact also implies that the dynamical systems \eqref{eq:mcgehee_w} and \eqref{eq:mcgehee} share not only the same equilibrium points $(0,\vt^*,\vt^*+k\pi)$, but also the same linearization. As a consequence, with the help of minor changes, the dynamical characterization provided in Section \ref{sec:st_man} naturally extends to the setting considered above and it is possible to reformulate Lemma \ref{lem:stable_manif}. 

On second thought, basically all the proofs contained in Section \ref{sec:var} strongly depend on the \emph{Lagrange-Jacobi inequality} (Lemma \ref{lem:lj}) and its consequences. In particular, when $W\equiv0$, the $-\al$-homogeneity of $V$ and the convexity of the inertial moment allow us to provide all the useful (upper or lower) estimates on the term $h+V$. Again, this can be reset in our new framework, since the hypotheses \eqref{hyp:V} on the perturbation $W$ tell us that we can recover a $-\al$-homogeneity on $V$ when $r$ is sufficiently small. Indeed, the term $r^\al W+r^{\al+1}|\nabla W|\to 0$ as $r\to 0$, so that, eventually choosing a smaller $r_{LJ}$ in Lemma \ref{lem:lj}, we can carry out again the entire argument. 

Finally, we want to remark that a complementary choice of $h\geq 0$ is not dramatic in this setting. In particular, if $W\equiv0$, this will induce the choice $r_{LJ}=+\infty$ and thus the presence of an infinite Hill's region, as expected in a parabolic or hyperbolic problem. On the other hand, if $W\not\equiv 0$, we could still have a bound on $r_{LJ}$, depending on the sign of $W$ close to the singularity.

\subsection{$\boldsymbol{d>2}$, $\boldsymbol{W\not\equiv 0}$, $\boldsymbol{h\in\R}$}

In this higher dimensional setting, the construction presented in Sections \ref{sec:coll_man}-\ref{sec:st_man} needs to be properly modified, in order to take into account the more abstract nature of this case. We want to make once more clear that the variational approach of Section \ref{sec:var}-\ref{sec:main} is not affected by tanking into account higher dimensions. Moreover, since the discussion of the previous paragraph on the lower order perturbations does not change for $d>2$, we will assume $W\equiv 0$. In order to face the dynamical complications, we will basically adopt the technique introduced by R. McGehee in \cite{McG1974}  (see also \cite{DevInvMath1978, DevProgMath1981, BTV}) in order to sketch a proof for Lemma \ref{lem:mcgehee}. As a starting point, for $x=x(t)\in\R^d$, introduce the new variables
\[
\begin{cases}
r(t)=|x(t)| \\
s(t)=r(t)^{-1}x(t) \\
v(t)=r(t)^{\al/2}\langle\dot{x}(t),s(t)\rangle \\
u(t)=r(t)^{\al/2}\pi_{T_s\mathbb{S}^{d-1}}\dot{x}(t),
\end{cases}
\]
where $\pi_{T_s\mathbb{S}^{d-1}}$ represents the orthogonal projection on the tangent space of $\mathbb{S}^{d-1}$, i.e.,
\[
\pi_{T_s\mathbb{S}^{d-1}}z=z-\langle z,s\rangle s,\quad\mbox{for every}\ z\in\R^d.
\]
In this way, slowing down the time with the time-rescaling
\[
dt=r^{1+\al/2}d\tau,
\]
solutions of \eqref{eq:an_kep} will be equivalent to solutions of 
\begin{equation}\label{eq:dyns}
\begin{cases}
r'=rv \\
v'=\frac{\al}{2}v^2+|u|^2-\al V(s) \\
s'=u \\
u'=-\frac{2-\al}{\al}vu-|u|^2s+\nabla_T V(s),
\end{cases}
\end{equation}
where the homogeneity gives $V(x)=r^{-\al}V(s)$ and $\nabla_T V(s)=\nabla V(s)-\langle\nabla V(s),s\rangle s=\pi_{T_s\mathbb{S}^{d-1}}\nabla V(s)$ is commonly known as the tangential gradient of $V$. The conservation of energy law \eqref{eq:energy} in this variables translates to
\[
\frac12\left(|u|^2+v^2\right)-V(s)=r^\al h,
\]
and defines the energy shell
\[
\mathcal{H}_h=\left\lbrace(r,v,s,u)\in(0,+\infty)\times\R\times\mathbb{S}^{d-1}\times T_s\mathbb{S}^{d-1}:\ \frac12(|u|^2+v^2)-V(s)=r^\al h\right\rbrace\simeq\R^{2d-1}.
\]
In $\mathcal{H}_h$ the variable $v$ reads
\[
v^\pm(r,s,u)=\pm\sqrt{2(V(s)+r^\al h)-|u|^2}.
\]
The choice of $v^-/v^+$ corresponds to the choice of studying in/outgoing trajectories to/from the singularity $r=0$. Indeed, equations \eqref{eq:dyns} can be reduced to a $(r,s,u)$-system, admitting $\{r=0\}$ as an invariant set. We denote by $\Lambda$ such a set, which is commonly known as \emph{collision manifold}, which actually is a smooth manifold of dimension $2d-2$ (see \cite[Proposition 1, pag.234]{DevProgMath1981}). Since we are interested in collision trajectories, we will take into account $v^-$, so that in $\Lambda$ system \eqref{eq:dyns} reads
\begin{equation}\label{eq:coll_man_d}
	\begin{cases}
	s'=u \\
	u'= \frac{2-\al}{\al}u\sqrt{2V(s)-|u|^2}-|u|^2s+\nabla_T V(s).
	\end{cases}
\end{equation}

From the linearization of \eqref{eq:coll_man_d}, it is possible to deduce the hyperbolicity of the equilibrium points of the $(r,s,u)$-system
\[
p^*=(0,s^*,0)\quad\mbox{such that}\ \nabla_TV(s^*)=0,
\]
as far as $V$ is a Morse function (see \cite[Proposition 4, pag. 237]{DevProgMath1981}). This leads to the existence of stable and unstable manifolds $\mathcal{W}^S$ and $\mathcal{W}^U$ for $p^*$, with $\dim\mathcal{W}^S+\dim\mathcal{W}^U=2d-1$. Again for our purpose of studying ingoing collision orbits, we naturally choose the $r$-eigenvalue to be negative so that (still following \cite{DevProgMath1981}, Lemma 5, pag.238) we infer that $\dim\mathcal{W}^S=d$ while $\dim\mathcal{W}^U=d-1$.

The hyperbolicity of $p^*$ gives rise to a local description of the manifold $\mathcal{W}^S$ as the graph of a $\mathcal{C}^2$-function in the variables $(r,s)$ (see \cite{Teschl_ode}, Theorem 7.3).

\bibliography{vivinabibliog}
\bibliographystyle{plain}

\end{document}